\documentclass[12pt]{amsart}

\usepackage{setspace}
\usepackage{amsmath, amssymb,amsthm, amsfonts}
\newtheorem{thm}{Theorem}[section]
\newtheorem{lma}{Lemma}[section]
\newtheorem{prop}{Proposition}[section]
\newtheorem{cor}{Corollary}[section]
\newtheorem{ques}{Question}

\newtheorem{definition}{Definition}

\theoremstyle{remark}
\newtheorem{remark}{Remark}[section]
\newtheorem{example}{Example}

\numberwithin{equation}{section}

\newcommand{\Ric}{\mbox{Ric}}
\newcommand{\R}{\mathbb R}

\newcommand{\be}{\begin{equation}}
\newcommand{\ee}{\end{equation}}
\newcommand{\bee}{\begin{equation*}}
\newcommand{\eee}{\end{equation*}}

\def\p{\partial}

\def\la{\langle}
\def\ra{\rangle}
\def\lf{\left}
\def\ri{\right}
\def\Pi{\displaystyle{\mathbb{II}}}
\def\Rm{\text{Rm}}
\def\dist{\text{dist}}
\def\tf{\tilde{f}}
\def\S{\Sigma}

\def\vh{\vspace{.2cm}}

\def\K{\mathcal{F}}

\def\l{\lambda}
\def\e{\epsilon}

\def\Int{\text{Int}}

\def\a{\alpha}

\def\th{\tilde{h}}

\def\AF{asymptotically flat }

\def\tf{\tilde{f}}

\begin{document}

\title[]{Static potentials on asymptotically \\
flat manifolds}

\author{Pengzi Miao$^1$}
\address[Pengzi Miao]{Department of Mathematics, University of Miami, Coral Gables, FL 33146, USA.}
\email{pengzim@math.miami.edu}
\thanks{$^1$Research partially supported by Simons Foundation Collaboration Grant for Mathematicians \#281105.}

\author{Luen-Fai Tam$^2$}
\address[Luen-Fai Tam]{The Institute of Mathematical Sciences and Department of
 Mathematics, The Chinese University of Hong Kong, Shatin, Hong Kong, China.}
 \email{lftam@math.cuhk.edu.hk}
\thanks{$^2$Research partially supported by Hong Kong RGC General Research Fund  \#CUHK 403108}

\renewcommand{\subjclassname}{
  \textup{2010} Mathematics Subject Classification}
\subjclass[2010]{Primary 83C99; Secondary 53C20}

\begin{abstract}
We consider the question whether a static potential on an asymptotically flat $3$-manifold
 can have nonempty zero set which extends to  the infinity.
 We prove that this does not occur if  the metric is asymptotically Schwarzschild with nonzero mass.
If the asymptotic assumption is relaxed to the usual assumption under which
the total mass is defined, we  prove that the static potential is unique up to scaling
unless the manifold is flat.
We also provide  some discussion concerning  the rigidity
of complete  asymptotically flat 3-manifolds without boundary  that  admit  a static potential.
\end{abstract}

\maketitle

\markboth{Pengzi Miao and Luen-Fai Tam}
{Static potentials on asymptotically flat manifolds}

\section{introduction}

In \cite{Corvino}, Corvino studied localized scalar curvature deformation of a Riemannian metric
and introduced the following definition:

\begin{definition}\label{df-static}
A Riemannian metric $g$ is called {\em static} on a manifold $M$ if the linearized scalar
curvature map  at $g$ has a nontrivial cokernel,
i.e. if there exists a nontrivial function $f$ on $M$ such that
\be \label{eq-static-c}
- (\Delta f) g  + \nabla^2 f - f \Ric  = 0 .
\ee
Here $\nabla^2$, $\Delta $ and $\Ric$ denote   the Hessian, the Laplacian and the Ricci curvature of $g$ respectively.
\end{definition}

We call   a nontrivial solution $f$ to \eqref{eq-static-c}   a {\it static potential} if it exists.
In \cite[Theorem 1]{Corvino}, Corvino proved that   if $(M, g)$ does not have a static potential,
one can deform the scalar curvature of $g$  through variations  having compact support in $M$.

It is known that a static metric  (as defined above) must have  constant scalar curvature (cf. \cite[Proposition 2.3]{Corvino}).
When this  constant   is zero (which is always the case for an asymptotically flat,  static metric),
 \eqref{eq-static-c}  becomes
\be \label{eq-static-f-s}
\nabla^2 f   = f \Ric
\  \ \mathrm{and} \ \ \Delta f  =  0 .
\ee
It is this  equation that explains the implications of Corvino's result  in   mathematical relativity,
where  a  vacuum static spacetime is  a $4$-dimensional Lorentz manifold that is isometric to
$ ( \R^1 \times M, - N^2 d t^2 + g ) $,
where $(M, g)$ is a $3$-dimensional Riemannian manifold,
$N  > 0 $ is a function on $M$, and the pair $(g, N)$
 satisfies
\be \label{eq-static-N}
\nabla^2 N = N \Ric  \ \ \mathrm{and} \ \
\Delta N = 0.
\ee
By \eqref{eq-static-f-s} and \eqref{eq-static-N}, one knows
if  $f$ is a static potential on a  manifold $(M, g)$ of  zero  scalar curvature,
  then  $( \R^1 \times \tilde{M},  - f^2 d t^2 + g )$
is a vacuum static spacetime, where $\tilde{M} = M \setminus f^{-1}(0) $.

There exists   a vast amount of literature concerning $3$-dimensional asymptotically flat  manifolds which
admit a {\em positive}  solution $N$ to \eqref{eq-static-N}  in the asymptotic region
(see  e.g. \cite{Bunting-Masood, Chrusciel98, Anderson,  Miao05, Beig-Schoen, C-G10, M-M13}).
Since the positivity of $N$ is always assumed  in these works,  it is natural to ask:

\begin{ques} \label{ques-main}
 Suppose $(E, g)$ is a $3$-dimensional asymptotically flat end on which there exists a static potential $f$.
 Under what conditions, is  $f$ free of zeros near infinity?
 \end{ques}

We recall the definition  of  an asymptotically flat $3$-manifold.

\begin{definition} \label{def-def-AF}
A Riemannian $3$-manifold $(M, g)$ (perhaps with boundary) is said to be  {asymptotically flat} if
there exists a   compact set $K$ such that  $ M \setminus K$ consists of
a finite number of components  $E_1$, $\ldots$, $E_k$,  called  the ends of $(M, g)$,
 such that each end $E_i$ is
diffeomorphic to $ \R^3 $ minus a ball  and,
under this diffeomorphism,
the metric $g$ on $E_i$  satisfies
\be \label{eq-AF-def-intro}
g_{ij}=\delta_{ij}+b_{ij} \mathrm{ \  with \ }
b_{ij} = O_2 ( | x|^{-\tau} )
\ee
for some constant $\tau>\frac12$. Here $x=( x_1, x_2, x_3 )$ denotes  the standard coordinate on $ \R^3$
and a function $\phi$ satisfies  $ \phi  = O_l ( | x|^{-\tau} )$ provided   $ | \p^i \phi | \le C |x|^{-\tau - i }$ for $ 0 \le i \le l $ and some constant $C$.
\end{definition}

We first describe a necessary condition for  $f$ to be positive near infinity.
On an end $E$ of an asymptotically flat  $(M, g)$,
suppose $f$ is a static potential  and $f>0$ near infinity, 
it is  known (cf. \cite{Beig-1980, Bunting-Masood}) that
there exists a  coordinate chart
$ \{ x_1, x_2, x_3 \}$  on $E$ near infinity   in which
the metric $g$ satisfies
\be \label{eq-AS-def-intro}
g_{ij} = \lf( 1 + \frac{m}{2 | x|} \ri)^4 \delta_{ij}+ p_{ij},
\ee
where  $ |p_{ij}|=O_2(|x|^{-2})$ and   $m$ is a constant that equals
the  ADM mass  (\cite{ADM61})  of $(M, g)$ at the end $E$.
Metrics satisfying the  fall-off  condition given in  \eqref{eq-AS-def-intro}
is often called  {\em asymptotically Schwarzchild} (AS).

Our main result in answering  Question \ref{ques-main} is that the AS
condition is also  a sufficient condition for the zero set $f^{-1}(0)$ to be  bounded, provided the mass is nonzero.

\begin{thm} \label{thm-main-AS}
Let $(M, g)$ be an asymptotically flat $3$-manifold with or without boundary.
If $g$ is asymptotically Schwarzschild on an end $E$ which has nonzero  mass,
then any static potential  $f$ on $E$  must be bounded  and is   either   positive or  negative
 outside a compact set.
\end{thm}

The  main tool in our proof of  Theorem \ref{thm-main-AS}  is Proposition \ref{prop-AF-static-zeroset}, which describes
 the asymptotic behavior of the zero set of $f$ assuming it is  unbounded. We also make use of an observation
 in Lemma \ref{lma-static-basic} (iii)
that  the Ricci curvature of $g$, when restricted to the zero set of $f$, is  a multiple of the induced metric.

In relation to the question of its positivity,  we also ask ``how many" static potentials may   exist.
We prove

\begin{thm} \label{thm-main-dimension}
Let $(M, g)$ be a connected,  asymptotically flat $3$-manifold with or without boundary.
Let $\K$ be the space of all solutions to \eqref{eq-static-f-s}.
Let  $\dim (\K) $ be  the dimension of $\K$.
Then
 $ \dim (\K) \le 1 $ unless $(M, g)$ is flat.
\end{thm}

In the proof of Theorem \ref{thm-main-dimension}, beside  Proposition \ref{prop-AF-static-zeroset},
we also use  a local result that describes the dimension of $\K$ on any open set. 
We prove the following result using some techniques by Tod \cite{Tod-2000}. 

\begin{thm}\label{thm-local-intro}
Let $(M, g)$ be a connected, $3$-dimensional Riemannian manifold of zero scalar curvature.
Let $\K$ be the space of static potentials on $(M, g)$.
Then
\begin{enumerate}
  \item [(i)] $\dim(\K)\le 2$ unless $(M,g)$ is flat.
  \item [(ii)] If there exist two linearly independent functions  $f_1, f_2\in \K$ such that 
  $f_1^{-1}(0)\cap f_2^{-1}(0) \neq \emptyset $, then $(M,g)$ is flat.
\end{enumerate}
\end{thm}

Our method in proving Theorem \ref{thm-main-AS} and Theorem  \ref{thm-main-dimension} also allow us to
obtain some rigidity  results for  complete, asymptotically flat manifolds without boundary  which admit a static potential.
For instance, a direct corollary of Theorem \ref{thm-main-AS}, Theorem \ref{thm-rigidity-bounded} in Section  4 and the
Riemannian positive mass theorem \cite{SchoenYau79, Witten81} is  that

\begin{cor}\label{cor-main-rigidity-bounded}
Let $(M, g)$ be a complete, connected,  asymptotically flat $3$-manifold without boundary.
Suppose  $(M, g)$ is asymptotically Schwarzschild at each end.
If there is  a static potential on $(M, g)$, then
 $(M, g) $ is   isometric to either the Euclidean space $(\R^3, g_0)$ or  a  spatial Schwarzschild manifold
 $ ( \R^3 \setminus \{ 0  \},  ( 1 + \frac{  m}{ 2 | x | }  )^4 g_0 ) $ with $m>0$.
 \end{cor}

After the initial draft of this paper was completed, we were  informed by  Piotr Chru\'{s}ciel
and  Greg Galloway  that  there  in fact exists a spacetime approach toward Question \ref{ques-main}.  
Namely, using results in  \cite{CG69} on Cauchy development, 
results in \cite{Moncrief76, FMM80, Chrusciel91} on 
vacuum KID development (also see \cite{BC97}),
results in \cite{CO81} concerning existence of boost-type domains, and in
 particular the result of Beig-Chru\'sciel  in \cite[Theorem 1.1]{Beig-Chrusciel-97}
 which excludes boost-type Killing vector fields under appropriate conditions, 
 Question 1 can also be {approached}  in the spacetime setting.  
 
 We deem this spacetime method  a very natural, important and physically motivated way to
understand the structure of the zero set of static potentials.
Comparatively,  our approach toward Question \ref{ques-main} 
is a purely initial data based method and    our method is   more
 elementary.

The organization of the paper is as follows.
In Section 2, we discuss local properties of static metrics and prove Theorem \ref{thm-local-intro}.
In Section 3, we analyze static potentials on an asymptotically flat end  and
prove Theorems \ref{thm-main-AS} and \ref{thm-main-dimension}.
In Section 4, we  provide some discussion of  rigidity questions for complete \AF $3$-manifolds
which admits a static potential.

The authors want to give deep thanks to  Piotr Chru\'{s}ciel and Greg Galloway
for introducing them to the spacetime approach mentioned above. The authors also want to  thank Justin Corvino, Marc Mars  
and Richard Schoen for  their  helpful comments  on  this work.

\section{Local properties of static metrics}  \label{section-local}

In this section,  we assume that
 $(M, g) $ is a $3$-dimensional, connected, smooth Riemannian manifold whose scalar curvature $R$ is  zero.
 By \eqref{eq-static-f-s}, a nontrivial function $f$   is a static potential  on $(M, g)$ if 
\be\label{eq-static-1}
 \nabla^2   f = f \Ric   .
\ee

In \cite{Tod-2000}, Tod studied the question when a spatial metric could
give rise to a static spacetime in more than one way.
In our work,
we often need to apply Proposition 2 (ii), Corollary 3 (i) and equation (15) in \cite{Tod-2000}.
 For convenience, we list these results of Tod in the next  Proposition. We also sketch the proof.

\begin{prop}[Tod \cite{Tod-2000}]  \label{prop-tod}
Let $ \{ e_1, e_2, e_3 \}$ be an orthonormal frame that diagonalizes the Ricci curvature at a given point $p$.
\begin{itemize}

\item[(i)]  Suppose  $ f $ is a static potential.  Then
\begin{equation*} \label{eq-tod}
\begin{split}
 f ( R_{33;1} - R_{31;3} ) = & \ ( R_{22} - R_{33} ) f_{;1}  \\
 f ( R_{11;2} - R_{12;1} ) =  & \ ( R_{33} - R_{11} ) f_{;2}  \\
 f ( R_{22;3} - R_{23;2} ) = & \ ( R_{11} - R_{22} ) f_{;3}  .
 \end{split}
 \end{equation*}

\item[(ii)] Suppose $ \{ R_{11}, R_{22}, R_{33 } \} $ are  distinct and  suppose $ N$, $V$ are two positive static potentials.
Then $ V = c N$ for some constant $c$.

\item[(iii)] Suppose $ R_{11} = R_{22} \neq R_{33}$ and suppose $ N$ is a positive static potential. If $ f $ is another static potential,
then $ Z = N^{-1} f $ satisfies
$ Z_{;1} = Z_{; 2} = 0 . $

\end{itemize}

\end{prop}

\begin{proof}
(i) Let $ \{ a, b, c, \ldots \}$ denote indices that run through $ \{ 1, 2, 3 \}$. Differentiating the static equation, one has
$$ f_{; abc} = f_{;c} R_{ab} + f R_{ab;c} . $$
Let $ R^d_{ \ a c b} $ be the curvature tensor. (In our notation, $ R^d_{ \ a c b} $ is  given by
$$  \nabla_{\p_c} \nabla_{\p_b} \p_a - \nabla_{\p_b} \nabla_{\p_c}  \p_a = R_{ \ a c b}^d \p_d  $$
in a local coordinate chart.)
Then
\be \label{eq-nabc}
\begin{split}
 R_{\ abc}^d f_{;d} = & \ f_{;abc} - f_{;acb} \\
= & \ f_{;c} R_{ab} - f_{;b} R_{ac}  + f ( R_{ab;c} -   R_{ac;b}).
\end{split}
\ee
In $3$-dimension,  the curvature tensor and the Ricci curvature  are related by
\be \label{eq-cur-ricci}
 R^{d}_{\ abc} =
\delta^d_b R_{ac} - \delta^d_c R_{ab} + g_{ac}R^d_b - g_{ab} R^d_c
+ \frac{1}{2} R ( \delta^d_c g_{ab} - \delta^d_b g_{ac} ) .
\ee
It follows from \eqref{eq-nabc}, \eqref{eq-cur-ricci} and the fact $ R=0$ that
\be \label{eq-ricif}
 2 ( f_{;b}R_{ac} - f_{;c} R_{ab} ) + g_{ac} f_{;d} R^d_b - g_{ab} f_{;d} R^d_c  =    f ( R_{ab;c} -   R_{ac;b}).
\ee
Take $ a = b \neq c$ and use the fact $ \{ e_1, e_2, e_3 \}$ diagonalizes $ \Ric$,  one has
\be \label{eq-nac}
 f_{;c} ( - 2 R_{aa}   -  R_{cc})   =    f ( R_{aa;c} -   R_{ac;a}).
\ee
Now (i)  follows from \eqref{eq-nac} and the fact $ R = 0 $.

(ii)  The assumption on $ \Ric$ implies that $ \Ric$ has distinct eigenvalues in an open set $U$.  Hence,
$ \nabla \log N  = \nabla \log V $  on $U$ by (i), which shows $V = c N $ for some constant $c$ on $U$.  Since $V$, $N$ are both harmonic functions,
$V = c N$ on $M$ by unique continuation.

(iii)  Apply (i) to $N$ and $ f = Z N$, one has
$$  ( R_{22} - R_{33} ) N Z_{;1}  =
 ( R_{33} - R_{11} ) N Z_{;2}  =
 ( R_{11} - R_{22} ) N Z_{;3}  = 0 . $$
The claim then follows from the fact $N \neq 0 $ and $ R_{11}= R_{22} \neq R_{33} .$
\end{proof}

The zero set of a static potential, if nonempty, was known to be a totally geodesic
hypersurface (cf. \cite[Proposition 2.6]{Corvino}).
In the next lemma, we give  more  geometric
 properties of  this  zero set.

\begin{lma}\label{lma-static-basic} Suppose $f$ is  a static potential with nonempty zero set. Let $\S=f^{-1}(0)$.
\begin{enumerate}
  \item [(i)] $\S$ is a totally geodesic hypersurface and
   $|\nabla f|$ is a positive constant on each connected component of $ \S$.
  \item [(ii)] At   any $p \in \S$, $\nabla f$ is an eigenvector of  $\Ric$.
\item [(iii)]  At any $ p \in \S$, let $\{ e_1,e_2, e_3 \}$ be an  orthonormal frame
that diagonalizes $ \Ric$ such that   $e_3$ is normal to $\S$.
Then
$R_{11} = R_{22} .$
\item[(iv)] Let $ K $ be the Gaussian curvature of $ \S$ at $ p$.
Using the same notations  in (iii),  one has
$ K  = 2 R_{11} = 2 R_{22}=-R_{33} .$
In particular, $K$ is  zero if and only if $ (M, g)$ is  flat at $p$.
\end{enumerate}

\end{lma}
\begin{proof} (i)
Let $ p \in \S$. If $ \nabla f(p)=0$, then along any geodesic $\gamma(t)$ emanating from $p$, $f(\gamma(t))$ satisfies $f''=\Ric(\gamma',\gamma')f$
and  $f(0)=f'(0)=0$.  This implies $f$ is zero near $p$. By unique continuation, $f = 0$ on $M$, thus a contradiction. Hence, $ \nabla f (p) \neq 0 $, which
implies  that  $\S$ is an embedded surface.
On $ \S$, the static equation shows   $ \nabla^2 f (X, Y) $ = 0  and $ \nabla^2 f (X, \nabla f ) = 0 $ for any tangent vectors  $X, Y$ tangential to $ \S$,
which readily  implies  that $ \S$ is totally geodesic and $ \nabla_X | \nabla f |^2 = 0 $.

(ii)  Since  $\S$ is totally geodesic,  it follows from the Codazzi equation that   $\Ric(\nu,X)=0$ for all $X$ tangent to $ \S$, where $\nu$ is the unit normal of $\S$.
Therefore, $ \nabla f = \frac{\p f}{\p \nu} \nu $ is an eigenvector of $ \Ric$.

(iii)  Apply  Proposition \ref{prop-tod} (i), one has
$$ ( R_{11} - R_{22} ) f_{;3}  = f ( R_{22;3} - R_{23;2} ) =  0  . $$
Since  $|f_{;3}|=|\nabla f|>0$, one concludes  $ R_{11} = R_{22}$.

(iv)   It follows from the  Gauss equation,  the fact $ R =0$ and (iii)  that $ K = - R_{33} = 2 R_{11} = 2 R_{22}$.
As a result, $ K = 0 \Leftrightarrow \Ric = 0 $ at $p$.

\end{proof}

In what follows, we  let   $ \K = \{ f \ | \ \nabla^2 f = f \Ric \}$.

\begin{lma} \label{lma-distinct}
If  the Ricci curvature of $g$  has distinct eigenvalues at a point, then $ \dim (\K) \le 1 $.
Here $ \dim (\K)$ denotes the dimension of $\K$. 
\end{lma}

\begin{proof}
The assumption on $ \Ric$  implies there is an open set $U$ such that
 $\Ric $ has distinct eigenvalues everywhere in $U$.
By  Lemma \ref{lma-static-basic} (iii),  a static potential $f$ is either positive or negative  in  $U$.
The claim now follows from  Proposition \ref{prop-tod} (ii).
 \end{proof}

Given two static potentials, if one of them is positive, one can look at their  quotient.

\begin{lma}\label{lma-static-quotient}
Suppose $ f$ and $N$ are two static potentials. Suppose $N$ is positive. Let  $Z=f/N$.
Then either $  Z  $ is a constant or  $\nabla Z$ never vanishes.
In the latter  case, one has
\begin{itemize}

\item[(i)]   each  level set of $Z$ is a  totally geodesic hypersurfaces.

\item[(ii)]  $ N^2 | \nabla Z |^2  $ equals a {constant}  on each connected component of the level set of $ Z$.

\item[(iii)]  $(M, g)$ is locally isometric to
$ \lf(  (-\epsilon, \epsilon) \times \Sigma, N^2 dt^2 + g_0 )  \ri) $
where $ \Sigma $ is a $2$-dimensional surface,  $ Z$ is a constant on each $ \Sigma_t = \{ t \} \times \Sigma$
and $ g_0 $ is a fixed metric on $ \Sigma$.
\end{itemize}

\end{lma}
\begin{proof} Let $ \{ x_i \}$ be local coordinates on $M$.
Since $N$ and $f=NZ$ both are solutions to  \eqref{eq-static-1}, we have
\bee
\begin{split}
NZR_{ij}
=& \ (NZ)_{;ij}\\
=& \ NZR_{ij}+NZ_{;ij}+N_{;i}Z_{;j}+N_{;j}Z_{;i}.
\end{split}
\eee
Therefore,
$ NZ_{;ij}=-N_{;i}Z_{;j}-N_{;j}Z_{;i} $
or equivalently
\be  \label{eq-Z-elliptic}
N \nabla^2 Z (v, w) = - \la \nabla N, v \ra \la \nabla Z, w \ra - \la \nabla N, w \ra \la \nabla Z, v \ra
\ee
for any tangent vectors $v$, $w$.

 Suppose $\nabla Z=0$ at some point $p$. Similar to the proof of Lemma \ref{lma-static-basic} (i),
 we consider an arbitrary geodesic $\gamma(t)$ emanating from $p$.
Taking $ v = w = \gamma'$ in  \eqref{eq-Z-elliptic}, we have
 $  N  Z(\gamma(t) ) '' =  - 2 N (\gamma(t))' Z(\gamma(t))' $.
 As $N > 0 $ and $ Z(\gamma(t))' |_{t=0} = 0 $, we have  $ Z(\gamma(t))' = 0 $, $ \forall \ t $. Hence
$Z$ is a constant near $p$.  By unique continuation \cite{Aronszajn-1957}, $Z$ is  a constant on $M$.

Next, suppose $ \nabla Z \neq 0$ everywhere. In this case, every level set  $Z^{-1}(t)$, if nonempty,  is an embedded hypersurface.
Let $ v$ and $w$ be  tangent vectors tangent to $Z^{-1} (t)  $,
 \eqref{eq-Z-elliptic} implies
$
N \nabla^2 Z ( v, w)  = 0 .
$
As  $N > 0 $ and
$
\nabla^2 Z (v, w) = \la \nabla_v (\nabla Z) , w \ra  = | \nabla Z | \Pi (v, w),
$
where $ \Pi ( \cdot, \cdot) $ is the second fundamental form of $ Z^{-1}(t) $ with respect to
$ \nu = \nabla Z / | \nabla Z |$, we have $ \Pi = 0 $. Hence $ Z^{-1} (t) $ is totally geodesic, which proves (i).

To prove (ii), let  $ v = \nabla Z $ and $  w $ be  tangent to $ Z^{-1}(t)$ in  \eqref{eq-Z-elliptic}, we have
$    N w \lf( | \nabla Z |^2 \ri) = - 2 w ( {N} )  | \nabla Z |^2$,
which implies
$  w \lf( N^2 | \nabla Z |^2  \ri) = 0 . $
Hence $ N^2 | \nabla Z |^2  $ equals a constant   on each connected component of $  Z^{-1}(t)$.

For  (iii),  let  $ X =  \nabla Z / | \nabla Z|^2 $ which is a nowhere vanishing vector field.
Given any point $ p \in M$, let $ \Sigma$ be a connected  hypersurface passing $p$ on which $ Z$ is a constant.
By considering the integral curves of $X$ starting from $\Sigma$ and shrinking $ \Sigma$ if necessary,
one knows  there exists  an open  neighborhood $U$ of $p$,
diffeomorphic to $ (-\epsilon, \epsilon) \times \Sigma $ for some $\epsilon > 0$,
on which the metric $g$ takes the form
$$ g = \frac{1}{| \nabla Z |^2}  d t^2 + g_t  $$
where $ \p_t =  X $,  $Z$ is a constant on each $\Sigma_t =  \{ t \} \times \Sigma $
and  $ g_t $ is the induced metric on $ \Sigma_t$.
Consider a background metric
$$
\bar{g} = d t^2 + g_t
$$
on $ U =  (-\epsilon, \epsilon) \times  \Sigma $.
Let $\Pi, \overline{\Pi}$ be the second fundamental form of $\Sigma_t $ in $(U, g)$, $(U, \bar{g})$ respectively with respect to $\p_t$. Then
$ \Pi = | \nabla Z | \overline{\Pi} . $
Since  $ \Pi = 0 $ by (i),  we have $ \overline{\Pi} = 0 $.
Hence $  \frac{d}{dt} g_t = 0 $ by the fact  $  \overline{\Pi} = \frac12 \frac{d}{dt} g_t $.
This shows, for each $t$, $g_t  = g_0$ which is  a fixed metric on $\Sigma$.
By (ii), $  N | \nabla Z |  $ is a constant on  $ \Sigma_t $.
Let  $ \phi (t) =  N | \nabla Z |   $.
Then
\bee
g = \frac{N^2}{ \phi (t)^2} d t^2 + g_0 .
\eee
Replacing  $ t $ by $ \int \frac{1}{\phi(t)} d t $, we have
$ g = N^2 d t^2 + g_0 $. This proves (iii).
\end{proof}

\begin{prop}\label{prop-Static-dim}  If $(M, g)$ is not flat at a point, then $\dim(\K)\le 2$.
\end{prop}
\begin{proof}
Suppose $\dim(\K)>2$. Let $f_1,f_2,f_3$ be three linearly independent static potentials.
 Let $ U  $ be an open set such that  $g $ is not flat at every point in $U$.  By Lemma \ref{lma-static-basic},  $U\setminus\cup_{i=1}^3f_i^{-1}(0)$ is nonempty. Hence
one can find a connected open set $V \subset U$ such that each $f_i$ is nowhere vanishing on $V$. Let $\{ \lambda_1,  \lambda_2,  \lambda_3 \}$ denote the eigenvalues of $\Ric$ in $V$.   $ \{ \lambda_1, \lambda_2, \lambda_3 \}$ can not be distinct by Proposition \ref{prop-tod}  (ii).
The fact $ g $ is not flat  and $ R  = 0$ shows $ \{ \lambda_1, \lambda_2, \lambda_3 \}$ can not be identical.
Therefore,
one may assume  $\lambda_1=\lambda_2 \neq \lambda_3$ in $V$.
Let $Z_1=f_1/f_3$, $Z_2=f_2/f_3$.
By Proposition \ref{prop-tod} (iii), both $\nabla Z_1$ and $\nabla Z_2$ are parallel to the
eigenvector of $ \Ric$   with eigenvalue $\lambda_3$.
Therefore,  at a point $q\in V$, $\nabla Z_1+\a\nabla Z_2=0$ for some constant $ \alpha$.  By Lemma \ref{lma-static-quotient}, $\nabla Z_1+\a\nabla Z_2\equiv0$ in $V$. So $Z_1+\a Z_2$ is a constant in $V$. Hence, $f_1+\a f_2=\beta f_3$ for some constant $\beta$, which is a contradiction.
\end{proof}

When the zero set of a given static potential is not empty, we can
 consider the behavior of another static potential along such a set.

\begin{lma} \label{lma-eq-on-zero}
Suppose $ f $ and $  \tf $ are two static potentials.
Suppose $\tf $ has nonempty zero set.
Let $ \Sigma = \tf^{-1}(0)$.
Then
\be \label{eq-f-on-zeroset}
 \nabla^2_\Sigma f  = \frac12 K f  \gamma
 \ee
along $ \Sigma$.  Here $ \nabla^2_\Sigma $ is the Hessian on $\Sigma$, $ \gamma $ is the induced metric on $ \Sigma$, and
 $K$  is the Gaussian curvature of $(\Sigma, \gamma)$.
 Consequently, $ K f^3$ equals a constant along each connected component of $ \Sigma$.

\end{lma}

\begin{proof}
By   Lemma \ref{lma-static-basic} (iii),
$ \Ric (X, Y) = \l  \gamma (X, Y)  $, $ \forall $
  $X$, $Y$ tangent to $ \Sigma$, where $ 2 \l +  \Ric(\nu, \nu) = 0$ and $ \nu $ is a unit normal to $ \Sigma$.
   Therefore,
 $ \nabla^2 f (X, Y) = f \l \gamma (X, Y) $ along  $ \Sigma$.
On the other hand,   $  \nabla^2 f (X, Y) =  \nabla^2_\Sigma  f (X, Y)$
since  $ \Sigma$ is totally geodesic.
Hence
$  \nabla^2_\Sigma  f  = f \l \gamma = \frac12 f K \gamma , $
where we have used $ K = 2 \l $ by  Lemma \ref{lma-static-basic} (iv).

 Let $\{ x_\alpha \}$ be local coordinates on $ \Sigma$.
 Taking divergence and trace of  \eqref{eq-f-on-zeroset},
 we have
\be \label{eq-div-and-trace}
(\Delta_\Sigma f )_{; \alpha} + K f_{; \alpha} = \frac12 ( K f  )_{;\alpha}  \ \ \mathrm{and} \ \ \Delta_\Sigma f  = K f
\ee
where $ \Delta_\Sigma$ is the Laplacian on $ (\Sigma, \gamma)$.
It follows from \eqref{eq-div-and-trace} that
\bee
K_{;\alpha} f + 3 K f _{; \alpha} = 0 ,
\eee
which implies
 $  ( K f^3)_{; \alpha} = 0 $.  Hence, $ K f^3 $ is a constant on each connected component of $ \Sigma$.
\end{proof}

To prove the main result in this section, we need an additional lemma in connection with Lemma \ref{lma-static-quotient} (iii).

\begin{lma} \label{lma-zero-curvature}
Suppose $(\Sigma_0 , g_0)$ is  a flat surface.
If $ \dim (\K)  \ge 2 $ on
$$(M, g) = ( ( - \epsilon, \epsilon) \times \Sigma, N^2 dt^2 + g_0 ) $$
where $N$ is a positive function on $M $  and $g$ has zero scalar curvature, then $(M, g)$ is flat.
\end{lma}

\begin{proof}
Take any $ (t, q) \in  (- \epsilon, \epsilon ) \times \Sigma $, the surface $ \Sigma_t = \{ t \} \times \Sigma$ has zero Gaussian curvature and
is totally geodesic in $(M, g)$.
 Let $ \{ e_1, e_2, e_3 \}$ be an orthonormal frame at $(t, q)$
which diagonalizes the Ricci curvature and satisfies $ e_3 \perp \Sigma_t$. Then $ R_{33} = 0 $ by the Gaussian equation.
Hence, $ R_{11} + R_{22} = 0 $. If $ R_{11} \neq R_{22}$, then $\Ric$ has distinct eigenvalues at $(t,q)$ and Lemma \ref{lma-distinct} implies $ \dim (\K) \le 1 $, contradicting to the assumption $ \dim (\K) \ge 2 $.
Therefore $ R_{11} = R_{22} = 0 $ by Lemma \ref{lma-static-quotient} (iii).
We conclude that $ g$ has zero curvature at $(t, q)$.
\end{proof}

\begin{prop}\label{prop-two-zeros}
Suppose  $ \dim (\K) \ge 2$. Let
 $f_1$ and $ f_2$ be two linearly independent static potentials.
Let $ P_1$, $ P_2$ be a connected component of
$f_1^{-1}(0)$, $  f_2^{-1}(0) $ respectively. If $P_1 \cap P_2 \neq \emptyset$, then
\begin{itemize}
\item[(i)]  $(M, g)$ is flat along $P_1 \cup P_2$.
\item[(ii)] $(M, g)$ is flat in an open set which contains $P_1 \setminus f^{-1}_2 (0) $ and $ P_2 \setminus f_1^{-1} (0)$.
\end{itemize}

\end{prop}
\begin{proof}
First we note that $ f_1^{-1}(0 ) \cap f_2^{-1} (0) $ is an embedded curve  (hence a geodesic  since both
$P_1$ and $P_2$ are totally geodesic).  This is because $ f_1 $ and $ f_2 $ are linearly independent,  which implies
 $\nabla f_1$ and $\nabla f_2$ are   linearly independent at any point in $ f_1^{-1}(0 ) \cap f_2^{-1} (0) $.

Now let  $ K_1 $, $ K_2$ be the Gaussian curvature of $ P_1 $, $P_2$ respectively.
By  Lemma \ref{lma-eq-on-zero},
$ K_1 f_2^3 = C$  for some constant $C$ on $P_1$ and
$ K_2 f_1^3 = D$ for some constant $D$ on $P_2$.
Since $ f_1 = f_2 = 0 $ on $P_1 \cap P_2$, we have $ C = D = 0 $.
As $P_1 \cap f_2^{-1}(0)$, $P_2 \cap f_1^{-1}(0)$ consists of embedded curves, we conclude $ K_1 = 0 $ on $P_1$
and $K_2 = 0 $ on $P_2$. Consequently $g$ is flat along $P_1 \cup P_2$ by Lemma \ref{lma-static-basic} (iv). This proves (i).

To prove (ii), let  $ p $ be an arbitrary point in $  P_1 \setminus f_2^{-1} (0)$, then  $ f_2  $ does not vanish in an open set $U$ containing $p$.
Consider $ Z = f_1 / f_2 $ on $U$. We have $Z = 0 $ on $P_1 \cap U$.
By Lemma \ref{lma-static-quotient} (iii),  there exists an open
neighborhood $W $ of $p$, diffeomorphic to $ (-\epsilon, \epsilon) \times \Sigma$, where $ \Sigma $ is a small piece of $P_1$ containing $p$,
and
$Z$ is a constant on each $ \{ t \} \times \Sigma$, such that on $W$
the metric $g$  takes  the form of
$$ g = f_2^2 dt^2 + g_0 $$
where $g_0$ is the induced metric on $ \Sigma$.
By  (i), $(\Sigma, g_0)$ has zero Gaussian curvature. Since $\dim (\K) \ge 2$ on $(W, g)$,  Lemma \ref{lma-zero-curvature} implies  that $ g$ is flat in $W$.
Similarly, we know $g$ is flat in an open neighborhood of  any point in $ P_2 \setminus f_1^{-1}(0)$.  Therefore, (ii) is proved.
\end{proof}

To end this section, we apply the analyticity of a static  metric  to improve Proposition \ref{prop-two-zeros}.
It is known that, if $ (M, g)$ admits a static potential $f$, then
$g$ is analytic in harmonic coordinates around any point $p$ with $f(p) \neq 0$ (cf. \cite[Proposition 2.8]{Corvino}).

\begin{thm} \label{thm-local-flat}
Suppose  $ \dim (\K) \ge 2$. Let
 $f_1$ and $ f_2$ be two linearly independent static potentials.
 If $ f_1^{-1} (0) \cap f_2^{-1} (0) $ is nonempty, then
 $(M,g)$ is flat.
\end{thm}

\begin{proof}
Let $ S =   f_1^{-1} (0) \cap f_2^{-1} (0)$.
Given  any  $ p \in M \setminus S$, either $f_1 (p) \neq 0 $ or $ f_2 (p) \neq 0 $, hence there exists an open set containing $p$ in which $g$ is analytic.
As  $ f_1 $ and $ f_2 $ are linearly independent, $ S$ is an embedded curve. In particular $ M \setminus S $ is path-connected. Therefore,
by Proposition \ref{prop-two-zeros} (ii),  we conclude that $g$ is flat  in $M \setminus S$, hence flat in $M$.
\end{proof}

\begin{remark}
We note that a much stronger analytic property of static metrics was shown
by Chru\'{s}ciel in \cite[Section 4]{Chrusciel-analytic}.
Theorem \ref{thm-local-flat} also follows  from  Proposition \ref{prop-two-zeros} and
the result of Chru\'{s}ciel in \cite{Chrusciel-analytic}.
\end{remark}

\section{static potentials on  an asymptotically flat end} \label{section-AF}

In this section, unless otherwise stated,
we assume that $M$ is diffeomorphic  to $ \R^3 \setminus B(\rho)$,
where $B(\rho)$ is an open  Euclidean ball centered at the origin with radius $\rho>0$,
and $g$ is a smooth metric on $ M$ such that
with respect to the standard coordinates $\{x_i \}$ on $ \R^3$,
$g$ satisfies
\be\label{eq-AF-def}
g_{ij}=\delta_{ij}+b_{ij} \  \ \mathrm{with} \ \
b_{ij} = O_2 (|x|^{-\tau})
\ee
for some constant  $\tau \in ( \frac12 , 1]$.  We also assume  that $ g $ has zero scalar curvature.

On such an $(M, g)$, a static potential $f$ is necessarily 
smooth up  to  $\p M$ by   \eqref{eq-static-c} and the assumption  that 
$g$ is smooth up  to $ \p M$  (cf. \cite[Proposition 2.5]{Corvino}).
The following lemma shows that at infinity $f$ has  at most  linear growth.

\begin{lma}\label{lma-linear-growth}
Suppose $f$ is a static potential on $(M, g)$. Then $f$ has at most linear growth,
i.e. there exists $ C > 0 $ such that $ | f (x) | \le C | x | $.
\end{lma}
\begin{proof}
Let $ \Rm$ denote the Riemann curvature tensor of $g$.
By the AF condition \eqref{eq-AF-def},
we have
\be\label{eq-curvaturedecay}
r^{2+\tau}|\Rm|=O(1)
\ee
where $ r = | x |$.
Therefore, given any  $\e>0$,  there is  $r_0>\rho$  such that
$$|\Rm|(x)\le \frac 12 \epsilon | x |^{-2} \le \e (d(x)+r_0)^{-2}$$
if  $|x|>r_0$. Here  $d(x)=\dist(x, S_{r_0} )$, where $S_{r_0}=\p B(r_0)$, the Euclidean sphere with radius $r_0$.
Given any $x$ outside $ S_{r_0}$, let $\gamma(t)$, $ t \in [r_0, T]$,  be a minimal geodesic  parametrized by arc length connecting  $x$ and $S_{r_0}$
with
 $ \gamma (r_0) \in S_{r_0} $ and $ \gamma (T) = x$.
Then $f(t)=f(\gamma(t))$ satisfies
$$ f''(t) =h(t)f(t) , $$
where $ h(t)  = \Ric (\gamma'(t), \gamma'(t) ) $ satisfies
$$|h(t)|\le \e t^{-2}. $$
Let $\a=\frac12(1+\sqrt{1+4\e})$ and $a=\sup_{S_{r_0} } (|f|+|\nabla f|)$.
Define  $w(t)= A  t^\a$, where  $A >0$ is chosen so that $ A r_0^\a>a$ and  $ A  \a r_0^{\a-1}>a$, then  $w(t)$ satisfies
$$
w'' (t) =\e t^{-2}w, \  |f (r_0) | < w (r_0)  \ \mathrm{and} \   |f' (r_0) |<w' (r_0) .
$$
Suppose $ | f (t) | > w (t) $ for some $ t \in [r_0, T]$.
Let
$$t_1=\inf  \{t\in [r_0,T] \ | \ |f (t) | >  w (t) \} .$$
Then $ t_1 > r_0 $ and $ | f(t_1) | = w(t_1)$.
 On $[r_0,t_1]$, we have
$$
|f''(t)|=|h(t)f(t)|\le \e t^{-2}w=w''(t).
$$
Therefore, $\forall \ t \in [r_0,t_1]$,
$$
-w'(t)+w'(r_0)\le f'(t)-f'(r_0)\le w'(t)-w'(r_0)
$$
which implies
$
-w'(t)<  f'(t) <w'(t)
$
because $ |f'(r_0)| < w'(r_0)$. Integrating again, we have
$$
 -w(t)+w(r_0)< f(t)-f(r_0)< w(t)-w(r_0),
$$
which shows  $ - w(t) < f(t) < w(t) $ because
$|f(t_0)|<w(t_0)$.
Therefore, $ | f (t_1) | < w(t_1) $, which is a contradiction.
Hence we have
\be \label{eq-alphagrowth}
|f(t)|\le A t^\a, \ \forall \ t .
\ee
Now  choose $\e$ such that $\a< 1+\frac\tau2 $. It follows from  \eqref{eq-curvaturedecay} and \eqref{eq-alphagrowth} that
$$ |f''(t)| =|h(t)f(t)|\le A|h(t)|t^{1+\frac\tau2}  $$
where $|h(t)|\le C_1 t^{-2-\tau}$ for some $C_1$ independent on $x$ and $t$.
This shows  $|f' (t)| \le C_2 $ for some constant  $C_2  $ independent on $ x$.
Hence
 $$ | f (x) | \le a + C_2  ( |x| - r_0 )  ,$$
which proves  that $ f $ has at most linear growth.
\end{proof}

Using Lemma \ref{lma-linear-growth}, we now present the following structure result 
for static potentials near infinity (cf.  \cite[Proposition 2.1] {Beig-Chrusciel-96} and Remark \ref{rmk-prop-3-1}).  

\begin{prop}\label{prop-AF-static-f}
Suppose  $f$ is  a static potential on $(M, g)$. Then
\begin{enumerate}
\item[(i)] there exists a   tuple $(a_1, a_2, a_3)$ such that
 \bee \label{eq-f-xi}
  f= a_1 x_1 + a_2 x_2 + a_3 x_3 + h
 \eee
 where  $h$ satisfies $  \p h  = O_1 ( | x |^{ - \tau} )$
and
\bee \label{eq-condition-h-0}
| h | =
\lf\{
\begin{array}{lc}
O( |x|^{1-\tau}) & \  \mathrm{when} \   \tau < 1, \\
 O ( \ln |x| ) & \  \mathrm{when}   \ \tau = 1.
 \end{array}
 \ri.
\eee

\item[(ii)]  $ (a_1, a_2, a_3) = (0, 0, 0) $ if and only if $f  $ is bounded. In this case,
either $f > 0 $ near infinity or $ f < 0 $ near infinity; moreover,
upon  rescaling,
\bee \label{eq-f-xi-1}
f=1-\frac{m}{|x|}+o(|x|^{-1})
\eee
for some constant $m$.
\end{enumerate}

\end{prop}

\begin{proof}
By   \eqref{eq-AF-def} and Lemma \ref{lma-linear-growth},
$ |\nabla^2 f|=|f\Ric|=O(r^{-1-\tau})$ where $ r = | x|$.
Let $\phi=|\nabla f|^2$, then
\be \label{eq-grad-phi}
|\nabla\phi|^2\le 4|\nabla^2f|^2 \phi \le C_1r^{-2-2\tau}\phi
\ee
for some constant $C_1$.
By considering $\phi$ restricted to a minimal geodesic emanating from the boundary,  as in the proof of Lemma \ref{lma-linear-growth},
it is not hard to see that  \eqref{eq-grad-phi}  implies  $\phi$ is bounded. Hence
\be \label{eq-p2-est}
|\p_{x_i}\p_{x_j}f|=|f_{;ij}+\Gamma_{ij}^k\p_{x_k}f|=O(r^{-1-\tau}),
\ee
where `` $;$ " denotes covariant derivative   and $\Gamma^k_{ij}$ are the Christoffel symbols.
It follows from \eqref{eq-p2-est} that, for each $i$,
$ \lim_{x\to\infty}\p_{x_i}f $ exists and is finite.  Let $  a_i =  \lim_{x\to\infty}\p_{x_i}f  $ and define
$\lambda=\sum_{i=1}^3a_ix_i $, then
 $$
 |\p_{x_i}\p_{x_j}(f-\lambda)|=|\p_{x_i}\p_{x_j}f|=O(r^{-1-\tau})
 $$
 and
 $
\lim_{x\to\infty}\p_{x_i}(f-\lambda)=0.
$
This  implies
$$|\p_{x_i} (f-\lambda)|=O(r^{-\tau}),
$$
which then  shows
\be
f - \lambda=
\lf\{
\begin{array}{lc}
O(r^{1-\tau}) & \  \mathrm{when} \   \tau < 1, \\
 O ( \ln r ) & \  \mathrm{when}   \ \tau = 1.
 \end{array}
 \ri.
 \ee
 Let $h=f-\lambda$. This proves (i).

To prove (ii), first  suppose   $a_1=a_2=a_3=0$.
 Let $ \tau' $ be any fixed constant with $ \tau > \tau' > \frac12$.
 Then $ | f | = | h | = O ( r^{1 - \tau'})$, hence
 $ |\nabla^2 f|=|f\Ric|=O(r^{-1- 2 \tau'}) $. This combined with  $|\p_{x_i}f|=O(r^{-\tau})$  implies
 $|\p_{x_i}\p_{x_j}f|=O(r^{-1-2\tau'}) ,$
which  in turns shows  $ |\p_{x_i}f|=O(r^{-2\tau'})$. Since $2\tau'>1$, we conclude  that
$f$ has a finite limit as $x \rightarrow \infty$.
In particular, $f$ is bounded.

Next,  suppose $f$ is bounded.
Then $ a_1, a_2, a_3 $ must be zero since $h$ grows slower than a linear function.
Moreover,  $ \lim_{x \rightarrow \infty} \phi = 0$ since  $ | \p_{x_i} f | = O ( r^{- \tau} )$.
Let $ \Sigma=f^{-1}(0) $.
 By Lemma \ref{lma-static-basic}(i),  $ \Sigma$
is an embedded  totally geodesic surface and $\phi$ is a positive constant on any
connected component of $\Sigma$.  We want to prove that $\Sigma$ is bounded.

Let $P$ be any connected component of $\Sigma$,
then $P$ must be bounded (hence compact),
for otherwise contradicting to the fact $ \lim_{x \rightarrow \infty} \phi = 0 $
and $ \phi $ is a positive constant on $P$.
Next, note that there is $R_0>0$ such that $\p B(R)$, $\forall \  R\ge R_0$,
has positive mean curvature in  $(M, g)$.  Therefore,  for each fixed $P$,
 $P \cap \{ | x | > R_0 \} = \emptyset $ by the maximum principle and
 the fact that $P$ is a compact  embedded minimal surface.
 Since $R_0$ is independent of $P$, this implies  $\Sigma\cap \{ | x | > R_0 \} = \emptyset$,
 therefore either  $ f > 0 $ or $f < 0 $ on $\{ | x | > R_0 \}$.

To complete the proof,  let $ a =  \lim_{x \rightarrow \infty} f$ (which was shown to exists).
 Since $\Delta f=0$, we have
$  f = a + {A}{|x|^{-1}} + o ( | x |^{-1} ) $
for some constant  $A$ (cf. \cite{Bartnik-1986}). We want to show $ a \neq 0 $.
Suppose $ a = 0 $.
By what we have proved, we may assume $f>0$ near infinity.
 Let $ R > 0 $ be a constant such that  $ f > 0 $ on $ S_R = \p B (R) $.
Let   $ \psi $ be a harmonic function outside $ S_{R}$  such that
 $ \psi = \inf_{S_{R}} f > 0 $ on $ S_R $ and  $ \lim_{x \rightarrow \infty } \psi = 0 .$
Then $ f \ge \psi  $ by the maximum principle.
Since $\psi $ behaves like the Green's function which has a decay order of   $ \frac{1}{|x|}$, we have  $ A > 0 $.
On the other hand, the assumption $ a = 0 $ implies    $ f = O ( {| x |^{-1} })$, hence $|\nabla^2f|=O(r^{-3-\tau})$.
Since $|\p_{x_i}f|=O(r^{-2\tau'})$, we have
$ |\p_{x_i}\p_{x_j}f|=O(r^{-3-\tau})+O(r^{-1-\tau - 2\tau'})$ which implies  $|\p_{x_i}f|=O(r^{-3\tau'})$. Iterating this argument
and using
the fact  $ \tau'$ can be chosen arbitrarily close to $\tau$,
we conclude  $ |\p_{x_i}\p_{x_j}f|=O(r^{-3-\tau})$ and $ | \p_{x_i} f | = O ( r^{-2 - \tau})$.
This together with  $ a = 0 $  shows $|f|=O(r^{-1-\tau})$, contradicting the fact $A>0$.
Therefore,  $a\neq 0$. Multiplying $f$ by a nonzero constant, we conclude
$ f=1-{m}{|x|^{-1}}+o ( |x|^{-1} ) $
for some constant $m$.  This complete the proof of (ii).
\end{proof}

\begin{remark} \label{rmk-prop-3-1}
 Proposition \ref{prop-AF-static-f}  was also stated in a more general setting  
 by  Beig and  Chru\'{s}ciel in \cite[Proposition 2.1] {Beig-Chrusciel-96} for  KID (Killing initial data).  
 The proof of \cite[Proposition 2.1] {Beig-Chrusciel-96} was briefly outlined  in Appendix C in \cite{Beig-Chrusciel-96}.  
 For the convenience of the reader, we have presented a detailed proof of Proposition \ref{prop-AF-static-f}. 
 \end{remark}

The next proposition describes the zero set of  a static potential $f$ near infinity in the case that $f$ is unbounded.

\begin{prop}  \label{prop-AF-static-zeroset}
Suppose $f$ is an unbounded static potential on $(M, g)$.
There exists a new set of  coordinates $\{ y_i \}$ on $  \R^3 \setminus B (\rho)$ obtained by
a rotation of $\{x_i \}$ such that, outside a compact set,
$ f^{-1} (0)  $
is given by the graph of a smooth  function $ q = q(y_2, y_3) $ over
$$ \Omega_C = \{ (y_2, y_3) \ | \ y_2^2 + y_3^2 > C^2 \}  $$
for some constant $ C>0$, where   $ q $ satisfies
\be \label{eq-condition-q}
    \p q   = O_1 (|\bar{y}|^{-\tau} )
\  \ \mathrm{and} \ \
| q | =
\lf\{
\begin{array}{ll}
O( | \bar{y}|^{1-\tau}) & \  \mathrm{when} \   \tau < 1 \\
 O ( \ln |\bar{y}| ) & \  \mathrm{when}   \ \tau = 1.
 \end{array}
 \ri.
\ee
Here $ \bar{y} = (y_2, y_3) $.
As a result,
if $\gamma_{_R} \subset f^{-1}(0) $ is the curve   given by
$$ \gamma_{_R} = \{ ( q( y_2, y_3) , y_2, y_3) \ | \ y_2^2 + y_3^2 = R^2 \}  $$
and $\kappa $ is the geodesic curvature of $ \gamma_{_R} $ in $f^{-1}(0)$, then
\be  \label{eq-g-curvature}
\lim_{R \rightarrow \infty} \int_{\gamma_{_R}}  \kappa = 2 \pi .
\ee
\end{prop}

\begin{proof}
Let $(a_1, a_2, a_3) $ and $h$ be given by  Proposition \ref{prop-AF-static-f} such that
$ f = \sum_{i=1}^3  a_i x_i + h $. As $f$ is unbounded,    $(a_1, a_2, a_3) \neq (0, 0, 0)$.
We can rescale $f$ so that $\sum_{i=1}^3 a_i^2 = 1 $.
Hence, there exists new coordinates $ \{ y_i \}$ obtained by a rotation of $ \{ x_i \}$  such that
\be  \label{eq-f-h}
f = y_1 + h (y_1, y_2, y_3)
\ee
where $h$ satisfies
\be \label{eq-h-y}
\p h = O _1( | y |^{-\tau} ) \ \ \mathrm{and} \ \
| h | =
\lf\{
\begin{array}{ll}
O( |y|^{1-\tau}) & \  \mathrm{when} \   \tau < 1 \\
 O ( \ln |y| ) & \  \mathrm{when}   \ \tau = 1 .
 \end{array}
 \ri.
\ee
It follows from \eqref{eq-f-h} and \eqref{eq-h-y} that
$$ \frac{\p f}{\p y_1} = 1 + \frac{\p h}{\p y_1}  =  1 + O ( | y |^{ - \tau } ) . $$
Therefore there exists a constant $C > 0 $ such that
$$ \frac{\p f}{\p y_1} > \frac12,  \ \ \forall \  (y_2, y_3) \in \Omega_{C} = \{ ( y_2, y_3) \ | \ | \bar{y} | > C \} . $$
For any fixed $ (y_2, y_3) \in \Omega_{C}$,  \eqref{eq-f-h} and \eqref{eq-h-y} imply
$$ \lim_{y_1 \rightarrow -\infty} f = - \infty, \ \ \lim_{y_1 \rightarrow \infty} f = \infty . $$
Hence the set
$   f^{-1}(0) \cap \{ (y_1, y_2, y_3) \ | \ (y_2, y_3) \in \Omega_{C}  \} \neq \emptyset $
and  is  given by the graph of some function $ q = q (y_2, y_3) $ defined on  $ \Omega_C$.
Since $ \nabla f \neq 0$ on $ f^{-1} (0)$, $ q $ is a smooth function by the implicit function theorem.
Given the constant $ C$, \eqref{eq-f-h} and  \eqref{eq-h-y} imply there exists another constant $ C_1 > 0 $ such that
$$ | f | \ge  \frac12 | y_1 | > 0  \ \ \mathrm{whenever} \ \ | \bar{y}| \le C \ \mathrm{and}  \ | y_1 | > C_1 . $$
Therefore,
\bee
\begin{split}
& \   f^{-1}(0) \cap \{ (y_1, y_2, y_3) \ | \ (y_2, y_3) \in \Omega_{C}  \}  \\
 = & \  f^{-1}(0) \setminus  \{ (y_1, y_2, y_3) \ | \ | y_1 | \le C_1,  \ | \bar{y} | \le C   \}.
\end{split}
\eee
This proves that, outside a  compact set,  $f^{-1} (0)$ is given by the graph of $ q $ over $ \Omega_C$.

Next we estimate $ q $ and its derivatives.
The equation
 \be \label{eq-of-q}
 q + h (q, y_2, y_3) = 0
 \ee
 and \eqref{eq-h-y}   imply that, if $|\bar{y}|$ is large,
\bee
\begin{split}
  | q |  =    | h (q, y_2, y_3)|
  \le
  \lf\{
  \begin{array}{ll}
  C_2 \lf(   |q | +  |\bar{y} | \ri)^{1-\tau} , & \tau < 1 \\
  C _2 \ln \lf( | q | + | \bar{y}  | \ri) , & \tau = 1
  \end{array}
  \ri.
\end{split}
\eee
for some constant $ C_2 > 0 $.
This in turn implies, as $ | \bar{y}|\rightarrow \infty$,
\bee \label{eq-q-1over2}
| q | = O  ( | \bar{y} |^{1-\tau} )  \ \mathrm{if} \ \tau < 1 \ \ \mathrm{and} \ \
| q | = O (  \ln { |\bar{y}|} ) \ \ \mathrm{if} \ \tau = 1.
\eee
Let $ \a, \beta \in \{ 2, 3 \}$.
Taking derivative of \eqref{eq-of-q}, we have
\be \label{eq-D-1}
\frac{ \p q}{\p y_\alpha  } = -  \frac{ \frac{ \p h}{\p y_\alpha } } {  1 + \frac{\p h}{\p y_1 } } = O ( | \bar{y} |^{ -\tau } ).
\ee
Similarly,  by taking derivative of \eqref{eq-D-1}, we have
$ \displaystyle \frac{ \p^2 q}{\p y_\beta  y_\alpha } = O ( | \bar{y}|^{-1-\tau} )$.

To verify \eqref{eq-g-curvature}, we   consider  the pulled back  metric
$ \sigma = F^* ( g) $ on $ \Omega_C$
where $ F: \Omega_C \rightarrow \R^3  $ is given by $ F  (y_2, y_3) = ( q (y_2, y_3) , y_2, y_3 ) $.
It follows from \eqref{eq-AF-def} and  \eqref{eq-condition-q}
that
\be \label{eq-decay-sigma}
\sigma_{\alpha \beta} = \delta_{\alpha \beta} + h_{\alpha \beta}
\ee
where  $ \sigma_{\alpha \beta} = \sigma (\p_{y_\alpha}, \p_{y_\beta} ) $ and
$ h_{\alpha \beta} $ satisfies
\be \label{eq-decay-tau}
| h_{\alpha \beta} | + | \bar{y} | | \p  h_{\alpha \beta} | = O ( | \bar{y} |^{-\tau} ) .
\ee
Direct calculation using \eqref{eq-decay-sigma} and \eqref{eq-decay-tau} then shows
\be \label{eq-kappa}
{\kappa} = R^{-1}  + O( R^{-1- \tau} )
\ee
while the length of $ C_{_R} $ is $ 2 \pi R + O (R^{1-\tau} )$.
From this, we conclude that \eqref{eq-g-curvature} holds.
\end{proof}

\begin{remark}
In \cite{Beig-Schoen}, Beig and Schoen solved static $n$-body problem
in the case that there exists a closed, noncompact, totally geodesic surface
disjoint from the bodies. One may compare Proposition \ref{prop-AF-static-zeroset}
with Proposition 2.1 in \cite{Beig-Schoen}.
\end{remark}

Now we are ready to prove the main results of this section.

\begin{thm}\label{thm-dim2}
Let $(M, g)$ be a connected, asymptotically flat $3$-manifold with or without  boundary.
If $\dim(\K) \ge 2$,  then $(M, g)$ is flat.
\end{thm}

\begin{proof}
It suffices to prove this result on an end of $(M, g)$.
So we assume $M$ is diffeomorphic to $\R^3$ minus an open ball.
Suppose $ f$ and $ \tf$ are two linearly independent static potentials.
We have the following three cases:

{\it Case 1}. Suppose both $f$ and $\tf$ are  bounded. By Proposition \ref{prop-AF-static-f} (ii), after rescaling,
we have
$$
f=1-\frac{m}{|x|}+o(|x|^{-1}), \ \ \tf=1-\frac{\tilde m}{|x|}+o(|x|^{-1})
$$
for some constants $m, \tilde m$.
Therefore, $ f - \tf $ is a bounded static potential satisfying
$f-\tf=-\frac{m-\tilde m}{|x|}+o(|x|^{-1})$. This contradicts  Proposition \ref{prop-AF-static-f} (ii).
Hence,  this case does not occur.

{\it Case 2}. Suppose $f$ is bounded and $\tf$ is unbounded.
By Proposition  \ref{prop-AF-static-f},  upon a rotation of coordinates and scaling,
we may assume that $\tf=x_1+h$, where $h$
 satisfies the properties in Proposition \ref{prop-AF-static-f}(i),
and  $f=1-\frac{m}{|x|}+o(|x|^{-1})$ for some constant $m$.
Let  $r_0>\rho$ be a fixed  constant such that
 $f>\frac12$ on $\{|x|\ge r_0\}$,
 and $ S_r = \p B(r) $ has positive mean curvature $\forall \  r \ge r_0$.
Let $\lambda_0>0$ be another constant such that if $\lambda>\lambda_0$,
$\tf_\lambda:=\tf-\lambda f$ will be negative on  $S_{r_0}$. For each $ \l > \l_0$,
let $ \Sigma_\lambda = \{ x \ | \ \tf_\lambda(x) = 0, \ | x | \ge r_0 \}$.
Then $ \Sigma_\l \neq \emptyset$ by Proposition  \ref{prop-AF-static-zeroset}.
As  $ \tf_\lambda < 0 $ on $ S_{r_0}$,    $ \Sigma_\l  $ does not intersect  $ S_{r_0}$.
Hence   $ \Sigma_\l$ is a surface without boundary.
Let $P  $ be any  connected component of $ \Sigma_\l$.
Since $(M, g)$ is foliated by positive mean curvature surfaces $\{ S_r \}$ outside $ S_{r_0}$
and   $P $ is  an embedded  minimal surface without boundary, $P  $ cannot be compact by the maximum principle.
 By  Proposition  \ref{prop-AF-static-zeroset}, we have
$P = \Sigma_\l $.
Let $K$ be the Gaussian curvature of $ \Sigma_\l$.
By Lemma \ref{lma-eq-on-zero}, $ K f^3 = C$  for some constant $C$ along $\Sigma_\l$.
Note that
$ \lim_{x \rightarrow \infty} K = 0  $
because  $ g$ is asymptotically flat  and $\Sigma_\l$ is totally geodesic.
This  implies $ C = 0 $ since $f$ is bounded.
Hence $ K f^3 = 0 $ on $\Sigma_\l$.
As $ f > 0 $ outside $S_{r_0}$, we conclude $ K = 0 $. Hence, $(M, g)$ is flat along  $\Sigma_\l$
by Lemma \ref{lma-static-basic}(iv).

Thus we have proved that $(M, g)$ is flat at every point in the set
$$
U=\bigcup_{\lambda>\lambda_0} \{ x \ | \ \tf(x)-\lambda f(x)=0, \ |x|>r_0\}.
$$
By the growth condition on $h$, we know that there exists  a constant $a>0$ such that
for all $x_1>a$ and all $(x_2, x_3) \in \R^2$ with $ x_2^2  + x_3^2 < 1 $, $$
{\tf(x_1,x_2,x_3)}>\lambda_0 {f(x_1,x_2,x_3)} > 0 .
$$
Clearly this implies that these points $(x_1,x_2,x_3)\in U$  and $U$ contains a nonempty interior.
Let $ \hat{M} = M \setminus ( f^{-1} (0) \cap \tilde{f}^{-1} (0) ) $. $ \hat{M}$ is either $M$ itself
or $ M$ minus an embedded curve, hence
$\hat{M}$ is  path-connected. Since $g$ is analytic on $\hat{M}$ which intersects $U$,
we conclude that $g$ is flat on $\hat{M}$,
hence flat everywhere in $M$.

 {\it Case 3}.
 Suppose both $ f$ and $ \tf$ are unbounded. By the proof of Proposition \ref{prop-AF-static-zeroset},
 upon a rotation of coordinates and scaling, we may assume
$  f  =   x_1  + h $,
$ \tf = a_1 x_1 + a_2 x_2 + a_3 x_3 + \th, $
where   $ h =O(|x|^{\theta})$, $\th =O(|x|^{\theta})$  for some constant $ 0 < \theta < 1$,
 and $ a_i$, $ i = 1, 2, 3$, are some  constants. Moreover, we can assume that
 $f^{-1}(0)$, outside a compact set,  is given by the graph of
$ q = q (x_2, x_3) $ where $ q = O ( | x_2 |^\theta + |x_3|^{\theta}) $.

Replacing $\tf$ by $\tf-a_1f$, we may assume $a_1=0$.
In this case, if $ a_2 = a_3 = 0 $, then Proposition \ref{prop-AF-static-f} (ii) implies
that $\tf$ is bounded and we are back to Case 2.
Therefore we may assume $(a_2, a_3) \neq (0, 0)$.
Without loss of generality, we can  assume  $a_2=1$ upon rescaling $\tf$
so that
$\tf =   x_2 + a_3 x_3 + \th.$
Given any large positive number $ a$, consider the point $ x_+ = ( q(a, 0), a, 0)$ which lies in $ f^{-1} (0)$.
We have
\be
\begin{split}
 \tf (x_+) = & \ a + \th (q(a,0), a, 0) \\
          = & \ a + O ( | a |^{\theta^2} + | a|^\theta ) .
 \end{split}
 \ee
Hence $ \tf (x_+) > 0 $ if $a $ is sufficiently large.  Similarly, we have $ \tf (x_-) < 0 $, where $ x_- = ( q( - a , 0) , - a, 0 )$,
for large $a$. Since $ x_+ $ and $ x_-$ can be joint by a curve that is contained in the graph of $q$, hence in $f^{-1}(0)$, we
conclude
$$ f^{-1} (0) \cap \tilde{f}^{-1} (0) \neq \emptyset .$$
Therefore $(M, g)$ is flat  by Theorem \ref{thm-local-flat}.
\end{proof}

\begin{thm} \label{thm-AS}
Let  $g$ be a smooth metric on $M = \R^3 \setminus B(\rho)$, where $B(\rho)$ is an open ball,
such that
\be \label{eq-AS}
 g_{ij} (x) = \lf( 1 + \frac{m}{2 |x |} \ri)^4 \delta_{ij} + p_{ij}
\ee
where $   p_{ij} (x) = O_2 ( | x |^{-2} )    $
and $m \neq 0 $ is a constant.
If $f$ is a static potential of $(M, g)$, then $f$ does not vanish outside a compact set.
\end{thm}

\begin{proof}
By Proposition \ref{prop-AF-static-f} (ii),  it suffices to prove that $f$ is  bounded.
Suppose  $f$ is  unbounded,  by Proposition \ref{prop-AF-static-zeroset}
   there exists a new set of coordinates $\{ y_i \}$, obtained by
a rotation of $ \{ x_i \}$,  such that the zero set of $f$ which we denote by $\Sigma$, outside a compact set,  is given by
the graph of a smooth  function $ q = q(y_2, y_3)$ defined on
$$ \Omega_C = \{ (y_2, y_3 ) \ | \
 \ y_2^2 + y_3^2 > C^2 \}$$
for some constant $C>0$. Here $ q $ satisfies \eqref{eq-condition-q} with $\tau = 1$.

Since $ \{ y_i \}$ differs from  $ \{ x_i \}$ only by a rotation,
the asymptotically Schwarzschild condition \eqref{eq-AS} is preserved  in the $\{ y_i \}$ coordinates, i.e.
\be \label{eq-AS-y}
 g_{ij} (y) = \lf( 1 + \frac{m}{2 |y |} \ri)^4 \delta_{ij} +  p_{ij}
\ee
where $   p_{ij} (y) = O_2 ( | y |^{-2} )    $.
The  Ricci curvature of $g$ now can be  estimated explicitly in terms of $y$.
By  \cite[Lemma 1.2]{Huisken-Yau1996},     \eqref{eq-AS-y} implies
\be \label{eq-Ricci-est}
\Ric(\p_{y_i}, \p_{y_j} ) = \frac{m}{ | y|^3 }  \phi (y)^{-2} \lf( \delta_{ij} - 3 \frac{ y_i y_j }{ | y |^2 } \ri) + O ( | y |^{-4} ) ,
\ee
where  $ \phi (y) = 1 + \frac{m}{ 2 | y |}$.

Given any $ \bar{y} = (y_2, y_3) \in \Omega_C$,  let  $y =  (q(\bar{y}), y_2, y_3)$ and
$ T_{ {y}} \Sigma  $  be the tangent space to $ \Sigma$ at $y$.
As a subspace in $ T_y \R^3$, $ T_{{y}} \Sigma  $ is spanned by
$$ v  =  (\p_{ y_2} q)  \p_{ y_1}  + \p_{ y_2} , \ w  =  (\p_{ y_3} q)  \p_{ y_1}  + \p_{ y_3}  .$$
Let $ | v|_g$, $| w|_g$ be the length of $v$, $w$  with respect to $g$   respectively.
Define  $ \tilde{v} = | v|_g^{-1} v,  \tilde{w} = | w|_g^{-1} w $,
we want to   compare
$$  \Ric (\tilde{v}, \tilde{v} ) \ \ \mathrm{and} \ \
\Ric( \tilde{w}, \tilde{w} )$$
when  $ | \bar{y} | $ is large.
By \eqref{eq-condition-q} and \eqref{eq-Ricci-est},  we have
\be \label{eq-Ric-v}
\begin{split}
 \Ric (v, v)
=  & \
 \frac{m}{ |y|^3} \phi (y)^{-2}  \lf[ 1 +  (\p_{ y_2} q)^2  - \frac{3}{|y|^2}  \lf[   (\p_{y_2} q) q +  y_2 \ri]^2    \ri] +  O ( |\bar{y}|^{-4} ) \\
  = & \  \frac{m}{ |y|^3}    \phi(y)^{-2} \lf( 1  - \frac{3 y_2^2 }{ |  y |^2 }   \ri) +  O ( |\bar{y}|^{-4} ) .
\end{split}
\ee
Similarly,
\be
\Ric (w, w)  =  \frac{m}{ |y|^3}   \phi(y)^{-2}     \lf( 1  - \frac{3 y_3^2 }{ |  y |^2 }   \ri) +  O ( |\bar{y}|^{-4} ).
 \ee
On the other hand, \eqref{eq-condition-q} and \eqref{eq-AS-y} imply
\be \label{eq-v-length}
| v |^2_g =  \phi(y)^4 + O ( | \bar y |^{-2 } ) , \
| w |^2_g =  \phi(y)^4 + O ( | \bar y |^{-2 } )  .
\ee
Therefore, it follows from \eqref{eq-Ric-v} -- \eqref{eq-v-length} that
\be \label{eq-differ-Ric}
\begin{split}
\Ric (\tilde{v}, \tilde{v})  - \Ric (\tilde{w} , \tilde{w} )
=  \frac{3 m }{\phi (y)^{6} }   \frac{  \lf(   y_3^2 -   y_2^2  \ri)   }{ | y |^5  } +  O ( |\bar{y}|^{-4 } ).
\end{split}
\ee
Together with \eqref{eq-condition-q},  this shows that there exists $ (y_2, y_3) $ such that
$ \Ric ( \tilde{v}, \tilde{v} )  \neq \Ric ( \tilde{w}, \tilde{w})$  when  $ | \bar y| $ is large.
For instance, let $y_2 = 0 $ and $ y_3 \rightarrow + \infty$, then
\be
\begin{split}
 | y_3 |^2 ( \Ric (\tilde{v}, \tilde{v})  - \Ric (\tilde{w} , \tilde{w} ) ) \longrightarrow 3 m \neq 0 .
\end{split}
\ee
This  is a contradiction to  Lemma \ref{lma-static-basic} (iii).
We conclude that $f$ must be bounded.
\end{proof}

\section{Rigidity of  static asymptotically flat  manifolds}

In this section, we consider a complete,  asymptotically flat $3$-manifold  without boundary,
with finitely many ends, on which there exists a static potential $f$.
Two basic examples  are

\begin{example}
The Euclidean space $ (\R^3 , g_0 )$. Here
$f = a_0 +  \sum_{i=1}^3 a_i x_i $ and $ \{ a_i \} $ are constants.
\end{example}

\begin{example}
A spatial Schwarzschild manifold with  mass $m > 0$, i.e. $ ( \R^3 \setminus \{ 0 \},  ( 1 + \frac{  m}{ 2 | x | }  )^4 g_0 ) $.
In this case, $ f = \frac{ 1 - \frac{ m}{ 2 |x|} }{ 1 + \frac{m}{ 2|x|} }  $.
\end{example}

A natural question is whether these are the only examples of such manifolds?
We start by  showing  that $f$ must have nonempty zero set unless the manifold is
$(\R^3, g_0)$.

\begin{lma}\label{lma-zero-exist}
Let $(M,g)$ be a complete, connected, asymptotically flat $3$-manifold without boundary.
If $(M, g)$ has a static  potential $f$,
then $f^{-1}(0)$ is nonempty unless $(M, g)$ is isometric to $(\R^3, g_0)$.
\end{lma}
\begin{proof}
By  Bochner's formula and  the static equation \eqref{eq-static-f-s},
\be \label{eq-MP}
\begin{split}
 \frac12 \Delta | \nabla f |^2
 = & \ | \nabla^2 f |^2 + f^{-1}  \nabla^2 f ( \nabla f , \nabla f ) \\
  = & \ | \nabla^2 f |^2 + \frac12  f^{-1}  \nabla f ( | \nabla f |^2 )
\end{split}
 \ee
wherever $f \neq 0$.
Suppose  $f^{-1}(0) $ is empty, then  Proposition  \ref{prop-AF-static-zeroset} implies  $f$ is bounded.
By Proposition \ref{prop-AF-static-f} (ii), $\lim_{x\to\infty}|\nabla f|=0$ at each end of $(M, g)$.
Hence there is  $p \in M$ such that $|\nabla f|^2(p)=\sup_M|\nabla f|^2.$
By  \eqref{eq-MP}  and the strong maximum principle, $|\nabla f|^2$ must be a  constant and hence is  identically zero.
Therefore, $f$ is a nonzero constant and $\Ric=0$ everywhere. This shows  $(M, g)$ is flat and hence  isometric to $(\R^3, g_0)$ by volume comparison as $(M, g)$ is asymptotically flat.
\end{proof}

In \cite{Bunting-Masood}, Bunting and Masood-ul-Alam  proved that if $(M,g)$ is an
asymptotically flat $3$-manifold with  boundary, with one end,  
on which there is  a static potential  $f $ which goes to $1$ at $\infty$ and is $0$  on $\p M$,
 then $(M, g)$ is isometric to a spatial Schwarzschild manifold with positive mass outside its horizon.
  By examining the proof in \cite{Bunting-Masood}, we observe that the result in \cite{Bunting-Masood} holds on  manifolds with any number of ends.

\begin{prop}\label{prop-B-M}
Let $(M,g)$ be a complete, connected, asymptotically flat $3$-manifold with nonempty   boundary, with possibly more than one end.
Suppose $ f$ is a static potential  such that $f>0$ in the interior and  $f=0$ on $\p M$. Then
$(M, g)$ is isometric to a spatial Schwarzschild manifold with positive mass outside its horizon.
\end{prop}

\begin{proof}
Since $f > 0 $ away from the boundary,  $f$ must be  bounded by Proposition \ref{prop-AF-static-zeroset}.
Upon scaling, we may assume $\sup_M f=1$.
Suppose $(M, g)$ has $k$ ends $E_1$, $\ldots$, $E_k$, $k \ge 1$.
For each $ 1 \le i \le k$,  Proposition \ref{prop-AF-static-f} (ii) implies $\lim_{x\to\infty, x\in E_i}f(x)=a_i$ for some constant $0<a_i\le 1$.
By the maximal principal,  $a_i = 1 $ for some $i$. Without losing generality, we  may  assume  $a_1=1$.

We proceed as in \cite{Bunting-Masood}.  Define  $\gamma^{+}=(1+f)^4g$ and $\gamma^{-}=(1-f)^4g$.
Then the following are true:
\begin{itemize}
\item  $\gamma^+$ and $\gamma^-$ have zero scalar curvature (cf. Lemma 1 in \cite{Bunting-Masood}).

\item If $ a_j = 1$, then $E_j$ is an asymptotically flat end in  $(M, \gamma^+)$ and the  mass of $(M, \gamma^+)$ at $E_j$ is zero;
on the other hand, $E_j$ gets compactified in $(M, \gamma^-)$ in the sense that if $p_j$ is  the point of infinity at $E_j$, then there is a $W^{2, q}$  extension
of $\gamma^-$ to $E_j \cup \{ p_j \}$  (cf. Lemma 2 and 3 in \cite{Bunting-Masood})

\item If $ a_j < 1 $, then clearly $E_j$ is an asymptotically flat end in both $(M, \gamma^+)$ and $(M, \gamma^-)$.

\end{itemize}
Glue $(M, \gamma^+)$ and $(M, \gamma^-)$ along $\p M $ to obtain  a manifold   $(\tilde{M} , \tilde{g})$,
then $ \tilde{g}$ is $C^{1,1}$ across $\p M $ in $\tilde{M}$ (cf. Lemma 4 in \cite{Bunting-Masood}).
Apply the Riemannian positive mass theorem as stated in \cite[Theorem 1]{Bunting-Masood} and use the fact that the mass
of $E_1$ in $(\tilde{M}, \tilde{g})$ is zero, we conclude that $(\tilde{M}, \tilde{g})$ is isometric to $ (\R^3, g_0)$.
In particular, this shows that $(M, g)$ only has one end.
The rest now follows from the main theorem in \cite{Bunting-Masood}.
\end{proof}

Proposition \ref{prop-B-M} can be used to answer  the rigidity question in the case that   $f$ is bounded.

\begin{thm} \label{thm-rigidity-bounded}
Let $(M,g)$ be a complete, connected, asymptotically flat $3$-manifold without boundary, with finitely many ends.
If there exists   a  bounded static potential  on $(M, g)$,
then $(M, g)$ is  isometric  to either $(\R^3, g_0)$ or  a  spatial Schwarzschild manifold
$ ( \R^3 \setminus \{ 0  \},  ( 1 + \frac{  m}{ 2 | x | }  )^4 g_0 ) $ with $m>0$.
\end{thm}

\begin{proof}   Let $f$ be a bounded static potential.
If $(M, g)$ has only one end, then $f$ must be a constant by Proposition \ref{prop-AF-static-f} (ii) and the fact $ \Delta f = 0 $.
Hence, $(M, g)$ is flat and is isometric to $(\R^3, g_0)$.

Next suppose  $(M, g)$  has more than one end,
in particular $(M, g)$ is not isometric to $(\R^3, g_0)$.
By Lemma \ref{lma-zero-exist},   $ f^{-1} (0) \neq \emptyset $.
By  Lemma \ref{lma-static-basic} (i) and Proposition \ref{prop-AF-static-f} (ii),    $ f^{-1}(0)   $ is a  closed totally geodesic hypersurface (possibly disconnected);
moreover $f$ changes sign near $f^{-1}(0) $.
Let $N_1$ be a component of   $\{f>0\}$, then
$N_1$ is  unbounded as $f = 0 $ on $\p N$.
Since $ f $ is either positive or  negative near the infinity of each end of $(M, g)$,
$N_1$ must be  asymptotically flat, with possibly more than one end, with nonempty  boundary $\Sigma$ on which  $f=0$.
By Proposition  \ref{prop-B-M} and \cite{Bunting-Masood},  $(N_1,g)$ is
isometric to
$ \lf( \{x\in \R^3 \ | \  |x|> \frac{m_1}{2} \}, \lf(1+\frac{m_1}{2 |x|}\ri)^4\delta_{ij} \ri)
$
with some constant $ m_1 > 0$.

Similarly, let  $N_2$ be the component of $\{f<0\}$ whose  boundary contains  $\Sigma$. By the same argument,  we know
that $(N_2,g)$ is isometric to $ \lf( \{y \in \R^3 \ | \  0<|y|< \frac{m_2}{2} \},  \lf(1+\frac{m_2}{2 |y|}\ri)^4\delta_{ij} \ri) $ for some $m_2 > 0 $.
Since $M$ is connected, we conclude that
 $M= N_1\cup N_2\cup \Sigma$.

Now we have $\Sigma=\{|x|=2m_1\}=\{|y|=2m_2\}$.
As the area of $\Sigma$ is given by $ 16 \pi m_1^2$ and $ 16 \pi m_2^2$ respectively, we have
$ m_1 = m_2 .$
This proves that  $(M,g)$ is isometric to a  spatial Schwarzschild manifold with positive mass.
\end{proof}

Next, we consider  the rigidity question  without the boundedness assumption of $f$.
We recall that, by Proposition \ref{prop-AF-static-f} (ii) and Proposition \ref{prop-AF-static-zeroset},
 the zero set of a static potential  on an asymptotically flat manifold  has only  finitely many components.

\begin{prop}  \label{prop-integration}
Let $(M,g)$ be a complete, connected, asymptotically flat $3$-manifold without boundary,  with finitely many ends $E_1$, $\ldots$, $E_k$.
Suppose there exists a static potential $f$  on $(M, g)$. Then
\begin{enumerate}
  \item [(i)]  $ \displaystyle  \int_M f|\Ric|^2=0 $.
  \item [(ii)]
  $  \displaystyle   \int_M |f|\,|\Ric|^2=4\pi\lf[ \sum_\a c_\a(\chi(\Sigma_\a)-k_\a)+\sum_\beta \tilde c_\beta\chi(\tilde \Sigma_\beta)\ri]  $.
Here   $ \{ \Sigma_\alpha \ |  \ 0 \le \alpha \le m   \} $ and  $ \{ \tilde \Sigma_\beta \ | \ 0 \le \beta \le n  \} $ are  the sets of  unbounded components and bounded components  of
$f^{-1} (0)$  respectively.    $ c_\a > 0$ and $\tilde{c}_\beta > 0$ are  the constants which equal $|\nabla f |$ on $\Sigma_\a $ and $\tilde{\Sigma}_\beta$ respectively.
For each $\a$, $k_\a \ge 1 $ is the number of ends $E_i$ with $E_i \cap \Sigma_\a \neq \emptyset$.
 $ \chi (\Sigma_\a)$ and $\chi (\tilde{\Sigma}_\beta )$ denote the Euler characteristic of $\Sigma_\a$ and $\tilde{\Sigma}_\beta$.
\end{enumerate}
\end{prop}

\begin{proof}
At each end $E_i$, $1 \le i \le k $, let $\{y_1, y_2, y_3\}$ be a set of coordinates in which $g$ satisfies \eqref{eq-AF-def}. If $f$ is unbounded in $E_i$, we  require
that $\{ y_1, y_2, y_3 \}$ be given by  Proposition \ref{prop-AF-static-zeroset}. For any large $r>0$, let $ S^i_r$ be the coordinate sphere $\{ |y|= r \}$ in $E_i$.
Let $ U_r $ be the region bounded by $S^1_r$, $\ldots$, $S^k_r$ in $M$.

By Lemma \ref{lma-linear-growth} and \eqref{eq-AF-def},  $|f| = O (r) $ and  $|\Ric|=O(r^{-2-\tau})$ in each $E_i$. Hence,
 the integrals in (i) and (ii) exist and are finite.
The static equation  \eqref{eq-static-f-s} implies
\be \label{eq-f-Ric-2}
 f | \Ric |^2 = \la \nabla^2 f , \Ric \ra .
 \ee
Integrating \eqref{eq-f-Ric-2} over $U_r$ and doing integration by parts, we have
\be \label{eq-f-Ric-2-Ur}
\begin{split}
\int_{U_r}f| \Ric |^2=& \sum_{i=1}^k \int_{S^i_r } \Ric (\nabla f, \nu)
\end{split}
\ee
where $\nu$ is the unit outward normal to $ S^i_r$ and we also have used the fact $g$ has zero scalar curvature.
Since $|\nabla f|$ is bounded by Proposition \ref{prop-AF-static-f},  $| \Ric |=O(r^{-2-\tau})$, and the area of $S_r^i$ is of order $r^2$, we  conclude that (i) holds  by letting
$ r \rightarrow \infty $ in \eqref{eq-f-Ric-2-Ur}.

To prove (ii),  we first choose $r$ sufficient large so that $\tilde \Sigma_\beta \subset U_r$, $ \forall \ \beta$.
If $f$ is unbounded, we assume it is unbounded in the ends $E_1$, $\ldots$, $E_l$, $ 1 \le l \le k $,  and    bounded in the other ends.
We then choose $r$ large enough so that outside each $S_r^i$ in $E_i$, $1 \le i \le l $, $f^{-1}(0)$ is  the graph of some function
$q = q(\bar{y})$ given by Proposition \ref{prop-AF-static-zeroset}; moreover, by \eqref{eq-condition-q} we can assume
the graph of $q( \bar{y})$ always intersects  $ S_r^i $ transversally.
Hence,  the set $U_r^+=U_r\cap \{f>0\}$ has Lipschitz boundary.
Integrating  \eqref{eq-f-Ric-2} over $U_r^+$  gives
\be\label{eq-unbounded-1}
\begin{split}
\int_{U_r^+} f | \Ric|^2 = & \int_{U_r \cap \lf(\cup_{\a=1}^m \Sigma_\a \ri)} \Ric (\nabla f, \nu) + \int_{\cup_{\beta = 0}^n \tilde  \Sigma_\beta }\Ric (\nabla f, \nu) \\
& \ +\int_{\p U_r   \cap \{f>0\}}\Ric (\nabla f, \nu) .
\end{split}
\ee
Here $ \nu$ denotes the outward unit normal to $\p U_r^+$.
As in (i),
\be\label{eq-unbounded-11}
 \lim_{r \rightarrow \infty} \int_{\p U_r  \cap \{f>0\}}\Ric (\nabla f, \nu)=0.
 \ee
 On each $\tilde \Sigma_\beta$ or $\Sigma_\a$,  by the fact $\nu= - \frac{ \nabla f}{|\nabla f|}$, we have
$$
\Ric (\nabla f, \nu) = - | \nabla f | \Ric (\nu, \nu)  = |\nabla f| K,
$$
where $K$ is the Gaussian curvature of $\tilde \Sigma_\beta$ or $\Sigma_\a$ by Lemma \ref{lma-static-basic} (iv).
Hence,
\be \label{eq-bounded-int}
\int_{ \cup_{\beta = 0}^n  \tilde  \Sigma_\beta}  \Ric (\nabla f, \nu )
 = 2 \pi \sum_{\beta = 0}^n  \tilde c_\beta \chi(\tilde \Sigma_\beta),
\ee
by the Gauss-Bonnet theorem,
and
\be \label{eq-unbounded-int}
\int_{ U_r \cap \lf(\cup_{\a=1}^m \Sigma_\a \ri) }
 \Ric (\nabla f, \nu )
=   \sum_{ \a = 0 }^m   {c}_\a \int_{{ U_r \cap \Sigma}_\a } K .
\ee
Note that  $\Sigma_\a$ is totally geodesic, hence \eqref{eq-AF-def} implies that $|K|$ decays on $\Sigma_\alpha$
in the order of $O(| y|^{-2 - \tau} )$ in  each  end $E_i $ with $ \Sigma_\a \cap E_i \neq \emptyset$.
But \eqref{eq-condition-q} implies that, on $ \Sigma_\a \cap E_i $,  $ |y|$ is equivalent to the intrinsic distance function to a fixed point in $\Sigma_\a$.
Therefore,
\be \label{eq-finite-K}
 \int_{\Sigma_\a} | K | < \infty .
 \ee
Let $ C_{_R}^i$ be the curve in $\Sigma_\a \cap E_i$ which is the graph of $q$ over
 the circle $\{ | \bar{y} | = R \}$ (see the definition of $C_{_R}$ in Proposition \ref{prop-AF-static-zeroset}).
 Let $\kappa$ denote the geodesic curvature of $C_{_R}^i$ in $\Sigma_\alpha$.
 By  the Gauss-Bonnet theorem and   Proposition \ref{prop-AF-static-zeroset},  we have
\be \label{eq-total-K-alpha}
\begin{split}
\int_{\Sigma_\a} K = & \ \lim_{R \rightarrow \infty}  \lf(  2\pi \chi (\Sigma_\a)  -  \sum_{i \in \Lambda_\a}   \int_{C^i_{_R} } \kappa  \ri) \\
= & \    2\pi \chi (\Sigma_\a)  - 2 \pi k_\alpha  ,
\end{split}
\ee
where  $ \Lambda_\a $ is the set of indices $i$ such that  $ \Sigma_\alpha \cap E_i \neq \emptyset $.
It follows from \eqref{eq-unbounded-int} -- \eqref{eq-total-K-alpha} that
\be \label{eq-GB-Sigma-a}
\lim_{ r \rightarrow \infty} \int_{ U_r \cap \lf(\cup_{\a=1}^m \Sigma_\a \ri) }
 \Ric (\nabla f, \nu )  =  2 \pi \sum_{\alpha = 0}^m c_\a   ( \chi (\Sigma_\a) - k_\a ) .
\ee
By \eqref{eq-unbounded-1} -- \eqref{eq-bounded-int} and \eqref{eq-GB-Sigma-a}, we conclude that
\be \label{eq-int-f-positive}
\int_{ \{ f > 0 \}  } f | \Ric|^2  = 2 \pi \sum_{\alpha = 0}^m c_\a   ( \chi (\Sigma_\a) - k_\a ) +  2 \pi \sum_{\beta = 0}^n  \tilde c_\beta \chi(\tilde \Sigma_\beta).
\ee
(ii) now follows from \eqref{eq-int-f-positive} and (i).
\end{proof}

\begin{remark}
From \eqref{eq-decay-sigma} and \eqref{eq-decay-tau}, one can  show
$ \lim_{r\to\infty}\frac{A(r)}{r^2}= \pi, $
where $A(r)$ is the area of $ D(r) \cap E_i$,  $i\in \Lambda_\a$, for  a geodesic ball $D(r)$  with radius $r$ in $\Sigma_\a$.
Therefore, the fact  $ \int_{\Sigma_\a} K = 2 \pi  ( \chi (\Sigma_\a) -  k_\a ) $ also follows from  results in
\cite{Hartman1964, Shiohama85}.
\end{remark}

Proposition \ref{prop-integration} implies  that $(M, g)$ must be $(\R^3, g_0)$ if  $M$  has simple topology.

\begin{thm} \label{thm-simple-top}
Let $(M,g)$ and $f$ be given as in Proposition \ref{prop-integration}.
If $M$ is orientable and   every  $2$-sphere  in $M$ is the boundary of a bounded domain,
 then $(M, g)$ is  isometric to $(\R^3, g_0)$.
 In particular, if  $M$ is homeomorphic to $\R^3$, then $(M, g)$ is isometric to $(\R^3, g_0)$.
\end{thm}

\begin{proof}
Suppose  $ \tilde{\Sigma}_\beta$ is  a compact component of $f^{-1}(0)$.
Since $M$ is orientable and $\tilde \Sigma_\beta $ is two-sided (with a nonzero normal $\nabla f$),
$\tilde{\Sigma}_\beta $ is orientable.
If  $\chi(\tilde{\Sigma}_\beta ) > 0$, then $\tilde{\Sigma}_\beta  $ is a $2$-sphere.
Hence $ \tilde{\Sigma}_\beta = \p \Omega$
for some bounded domain $\Omega$ in $M$ by the assumption.
This implies  $f\equiv0$ in $\Omega$ by the maximum principal  and therefore $f\equiv0$ in $M$ by unique
continuation \cite{Aronszajn-1957}. Thus,
$(M, g)$ is flat and is isometric to $(\R^3, g_0)$.
However, $(\R^3, g_0)$ does not contain any closed minimal surface.
Hence, we must have $\chi ( \tilde{\Sigma}_\beta )  \le 0 $ for all compact components $ \tilde{\Sigma}_\beta $
of $f^{-1}(0)$ if such a component exists.  On the other hand,
if $\Sigma_\a$ is a noncompact component of $f^{-1}(0)$, then  $\chi(\Sigma_\a)\le 1$.
By Proposition \ref{prop-integration} (ii), we have
\bee
  \int_M |f|\,|\Ric|^2\le 0.
  \eee
This implies $\Ric \equiv 0 $ and therefore $(M, g)$ is isometric to $(\R^3, g_0)$.
\end{proof}

In what follows,  we replace the topological assumption  in  Theorem \ref{thm-simple-top}
by an assumption that $f$ has no critical points.
For this purpose,   we analyze the behavior of   integral curves of the gradient of a static potential.
We formulate the results in a setting similar to that in Proposition \ref{prop-B-M}.

\begin{prop} \label{prop-integral-curve}
Let $(M,g)$ be a complete, connected, asymptotically flat $3$-manifold with nonempty boundary,
with  finitely many ends  $E_1$, $\ldots$, $E_k$.
Suppose there exists a static potential  $f$  with $f=0$ on $\p M$.
Given any point $p \in \Int(M)$,  the interior of $M$,  let $\gamma_p(t)$ be the integral curve of $\nabla f$ with $\gamma_p(0)=p$.
Let $(\a,\beta)$ be the maximal interval  of existence of $\gamma_p$ inside  $\Int(M)$.

\begin{enumerate}
\item[(a)]   If $\beta <\infty$, then  $\lim_{t\to\beta}\gamma_p(t)=x$  for some $x\in \p M$;
if $\a>-\infty$, then $\lim_{t\to\a}\gamma_p(t)=y$ for some $y\in \p M$. Consequently,  either $\a=-\infty$ or $\beta=\infty$.

    \vh

\item[(b)]  If $\beta=\infty$, then $\lim_{t\to\infty}f(\gamma_p (t))=b>-\infty$. Moreover,

\vh

    \begin{itemize}

   \item[(i)] if $b=\infty$, then, as $t \to \infty$, $ \gamma_p(t) $  tends to infinity in an end $E_i$ on which  $f$ is unbounded;
   \vh
   \item[(ii)]  if $b<\infty$, then $ b \neq 0 $  and $\lim_{t\to\infty}|\nabla f|(\gamma_p (t))=0$;
   \vh
       \item[(iii)] if $b<0$, then
    $    \bigcap_{t>0}\overline{\{\gamma_p (s )|  \ s >t\}}\neq\emptyset $
    and consists of critical points of $f$.
    \end{itemize}

  \vh

\item[(c)]  If $\a=-\infty$, then $\lim_{t\to-\infty}f(\gamma_p (t))=a<\infty$. Moreover,
\vh
    \begin{itemize}
    \item[(i)] if $a=-\infty$, then,  as $t \to - \infty$, $ \gamma_p(t) $  tends to infinity in an end $E_i$ on which  $f$ is unbounded;
    \vh
         \item[(ii)]  if $a>-\infty$, then $ a \neq 0 $  and $\lim_{t\to-\infty}|\nabla f|(\gamma_p (t))=0$;
         \vh
             \item[(iii)] if $a> 0$, then $  \bigcap_{t<0}\overline{\{\gamma(s )|  \ s <t\}}\neq\emptyset $
    and consists of critical points of $f$.
\end{itemize}

\end{enumerate}
\end{prop}

\begin{proof}
If $p$ is a critical point of $f$, then $\gamma_p(t)=p$, $\forall \  t \in (-\infty,\infty)$.
Also  $f(p)\neq0$ by Lemma \ref{lma-static-basic} (i).
The proposition is obviously true in this case.
In the following, we assume $ \nabla f (p) \neq 0$.
Then   $\nabla f(\gamma_p(t))\neq 0$ for all $t$ and
\be\label{eq-integralcurve-1}
\frac{d }{dt}f(\gamma_p (t))=|\nabla f|^2(\gamma_p (t))>0.
\ee
By Proposition \ref{prop-AF-static-f}, $ \lim_{x \rightarrow \infty} | \nabla f |$ exists and is finite at each end $E_i$.
Therefore,
\be\label{eq-integralcurve-2}
 | \nabla f | (x) < B, \ \forall \ x \in M
\ee
for some  constant  $B >0 $.  Suppose $\beta<\infty$, then for $t_2> t_1>0$,
\bee
d(\gamma_p (t_1),\gamma_p (t_2))  \le  \int_{t_1}^{t_2}|\gamma_p'(s)|ds
\le   \ (t_2-t_1) B,
\eee
where $ d (\cdot, \cdot)$ denotes the distance on $(M, g)$.
Hence $\lim_{t\to\beta}\gamma_p (t)= x$ for some $x \in  M$.
Since $(\a,\beta)$ is the maximal interval of existence of $\gamma_p(t)$ in $\Int(M)$,
we conclude $x\in \p M$.  Similarly, if $\a>-\infty$, then $\lim_{t\to\beta}\gamma_p(t)=y$, for some $y\in \p M$.
If $\a>-\infty$ and $\beta<\infty$, then $f(x)=0=f(y)$, which contradicts \eqref{eq-integralcurve-1}. This proves  (a).

To prove (b), we note that  \eqref{eq-integralcurve-1} implies  $f(\gamma_p(t))$ is increasing, hence
 $\lim_{t\to\infty}f(\gamma_p (t)) = b $ exists and $ b > - \infty$.
If  $b=\infty$, then
there exists  $t_n\to\infty$ such that $\gamma_p(t_n)\to\infty$ in some end $E_i$ on which $f$ is unbounded.
Let $\{ t_n' \} $ be any other sequence with $ t_n' \to \infty$.
We claim that $\gamma_p (t_n')$ must  tend to infinity in $E_i$ as well.
Otherwise, passing to subsequence, we may assume that $\gamma_p (t_n')$ tends to infinity in another end $E_j$ with $ j \neq i$.
But this implies that, for large $n$,  there exists $t_n'' $ between $t_n$ and $t_n'$ such that $ \gamma_p (t_n'')$ lies in a fixed compact set  $K$ of $M$
(for instance the set $K$ used in Definition \ref{def-def-AF}). This  contradicts  the fact $ \lim_{n \rightarrow \infty} f (\gamma_p ( t_n'')) \to b = \infty$.
Therefore, $ \gamma_p (t)$ tends to infinity in $E_i$ as $ t \rightarrow \infty$, which  proves (i) in (b).

Next, suppose $b<\infty$.  Let $\{ t_n \} $ be any sequence such that  $ t_n \to\infty$.
Given any fixed number   $ 0  < \delta <  \frac{1}{B}  $, we have
\bee
\int_{t_n-\delta}^{t_n+\delta}|\nabla f|^2(\gamma_p (t))dt
=  f(\gamma_p(t_n+\delta))-f(\gamma_p(t_n-\delta)) \rightarrow  0, \ n \to \infty .
\eee
Hence   there exists   $t_n'\in[t_n-\delta,t_n+\delta]$  such that $|\nabla f|(\gamma_p (t_n') )\to 0$.
Define $B_{\gamma_p (t_n) }  (1) = \{ q \in M \ | \ d (q , {\gamma_p (t_n) } ) < 1 \} $.
For large $n$,   \eqref{eq-integralcurve-2} implies   $|f|<2 |b|+2B$ on $B_{\gamma(t_n))}(1)$.
This together with the fact    $ \nabla^2 f = f \Ric $ and $(M, g)$ is asymptotically flat implies
 \be \label{eq-Hessf-bd}
 |\nabla^2f| \le C_1
 \ee
  on  $B_{\gamma(t_n))}(1)$ for some  constant $C_1$  independent on $n$ and $\delta$.
Now let $\phi = | \nabla f |^2$, then $ \nabla \phi $ is dual to the $1$-form
  $ 2 \nabla^2 f (\nabla f , \cdot)  . $
  By \eqref{eq-integralcurve-2} and \eqref{eq-Hessf-bd},
  we conclude
  $$ | \nabla \phi |  \le C_2 $$   on $B_{\gamma(t_n))}(1)$ by a constant $C_2$ independent on $n$ and $\delta$.
Note that  $d(\gamma(t_n),\gamma(t_n'))\le \delta B < 1 $, we therefore have
$$
\phi (\gamma_p(t_n))\le \phi(\gamma_p (t_n'))+2\delta B C_2.
$$
Since $\phi (\gamma_p (t_n'))\to0$ and $\delta$ can be arbitrarily chosen, we conclude that $\phi (\gamma_p(t_n)) \to 0 $ as $n\to\infty$.

We also want to show $b\neq0$. Let $\{ t_n \}$ be given as above.
Suppose $ \{ \gamma_p (t_n) \} $ is unbounded, then  passing to a subsequence we may assume
$\gamma_p (t_n) \to \infty $  in  some end $E_j$.  If $f$ is unbounded in $E_j$,  we would have
$|\nabla f|(\gamma_p (t_n) ) \ge C_3$ for some $C_3>0$ independent of $n$ by Proposition \ref{prop-AF-static-f} (i),
contradicting to  the fact $ | \nabla f | (\gamma_p (t_n) ) \rightarrow 0$.
Hence, $f$ is  bounded in $E_j$. By  Proposition \ref{prop-AF-static-f} (ii), we have  $b = \lim_{ x \rightarrow \infty, \ x \in E_j} f \neq0$.
Next,  suppose $\{ \gamma_p (t_n ) \}$ is bounded. Passing to a subsequence, we may assume
$\gamma_p (t_n)=q \in M$.  Then $q$ is a critical point of $f$ since $ | \nabla f | (\gamma_p (t_n) ) \rightarrow 0$.
Therefore, $b=f(q)\neq 0$ by Lemma \ref{lma-static-basic} (i). This completes the proof of (ii) in (b).

To prove (iii) of (b), it is sufficient to prove that if $b<0$ and if $\{ t_n \}$ is a sequence tending to   $\infty$,
then  $\{ \gamma_p (t_n) \}$ must be bounded, hence containing  a subsequence converging   to a critical point in $M$.
Suppose $\{ \gamma(t_n) \} $ is unbounded,  then passing to a subsequence we may assume $ \gamma(t_n) \to \infty$
in an end $E_j$ where  $f$ is bounded by the proof in (ii)  above. On $E_j$,
Proposition \ref{prop-AF-static-f} (ii) implies
\be \label{eq-f-expan-4-p}
 f = b - \frac{A}{|x|} + o ( | x |^{-1} ) , \ |x| \rightarrow \infty
\ee
 where   $ A$ is a  constant such that
\bee
\frac{A}{b} = m
\eee
which is the mass of $(M, g)$ at the end of $E_j$ (cf. \cite{Beig-1980, Bunting-Masood}).
By the positive mass theorem \cite{SchoenYau79, Witten81},
we have $ m > 0 $ (which can be seen by reflecting $(M, g)$ through $\p M$ since $\p M$ is totally geodesic).
Therefore,    $A < 0 $ because $b<0$.
As a result, $ f (\gamma_p (t_n) ) > b $ for large $n$ by \eqref{eq-f-expan-4-p}.
 But this leads to a contradiction to the fact that $ b = \lim_{ n \rightarrow \infty} f (\gamma_p (t_n) ) $ and
  $f(\gamma_p (t))$ is strictly increasing in $t$.   Therefore, $\{ \gamma_p (t_n) \}$ must be bounded.
This  proves (iii) of (b).

Claim (c) follows from (b) by replacing $f$ by $-f$.
\end{proof}

Using Proposition \ref{prop-integral-curve}, we obtain an analogue of Proposition \ref{prop-B-M}
with the assumption  $f > 0 $ replaced by 
that $f$ has no critical points.

\begin{cor} \label{cor-BM-no-critical-pts}
Let $(M,g)$ be a complete, connected, asymptotically flat $3$-manifold with nonempty   boundary $\p M$, with finitely many ends.
Suppose there exists a static potential $ f$ without critical points such that  $f=0$ on $\p M$.
 Then
  $(M, g)$ is isometric to a spatial Schwarzschild manifold with positive mass outside its horizon.
\end{cor}
\begin{proof}
By definition, $\p M$ is compact.  Let $\Sigma$ be a component of $\p M$.
 Since $\nabla f \neq 0 $ at $\Sigma$ by  Lemma \ref{lma-static-basic} (i),
 we may assume that
  $\nabla f$ is inward pointing at $\Sigma$. Consider the map
$F:\Sigma\times(0,\infty) \rightarrow \Int (M)  $ given by $F(x,t)=\gamma_x(t)$ which is the integral curve of $\nabla f$
such that $\gamma_x(0)=x\in \Sigma$. By  Proposition \ref{prop-integral-curve} (a), $\gamma_x$ is defined on $[0,\infty)$.
The fact $f=0$ and $ \nabla f \neq 0 $ at  $\Sigma $ implies that  $F$  is one-to-one.
Hence, by the invariance of  domain, the image $N$ of $F$ is open in $\Int(M)$. We want to prove that $N$ is also closed in $\Int(M)$.

Let $y \in \Int (M) $ be a point that lies in the closure  of $N$ in  $\Int(M)$.
Then there exist $x_i\in \Sigma$ and $t_i > 0$ such that $\tilde x_i=\gamma_{x_i}(t_i)$ converge to $y$.
Passing to a subsequence, we may assume that $x_i\to x\in \Sigma$ and $t_i\to a$  with $0\le a\le\infty$.
We claim that $a<\infty$.
If this is true, we will have  $y = \lim_{ i \rightarrow \infty} \gamma_{x_i} (t_i)  = \gamma_x (a) \in N$.
Suppose $a=\infty$.  Consider the integral curve $\gamma_{\tilde x_i}(t)=\gamma_{x_i}(t+t_i)$, which is defined on $(-t_i, 0] $. Let $\gamma_y(t)$ be the integral curve of $\nabla f$ with $\gamma_y(0)=y$.
Since $ t_i \rightarrow \infty$,  $ \{ \gamma_{\tilde x_i}(t) \} $ converge uniformly to $\gamma_{y}(t)$ on $[- n ,0]$ for any  $n>0$.
In particular, $\gamma_y(t)$ is defined on $(-\infty,0]$.
On the other hand, $f(\gamma_{x_i}(t))$ is strictly increasing in $t$ for all $i$. Hence, $f(\gamma_{\tilde x_i}(t))>0$ on $(-t_i,0]$,
which implies $f(\gamma_y(t))\ge 0$ on $(-\infty,0]$.   By Proposition \ref{prop-integral-curve} (c),
there exists  a critical point of $f$ in $M$, contradicting the assumption that $f$ has no critical points.

Therefore, $N$ is  closed in $\Int(M)$ and hence $N=\Int(M)$. Since $f > 0 $ along each  $\gamma_x (t) $ on $(0 , \infty)$,
we conclude that  $f>0$ in $N=\Int(M)$. Hence,  $(M, g)$ is isometric to a spatial Schwarzschild manifold with positive mass outside its horizon
 by \cite{Bunting-Masood} or Proposition \ref{prop-B-M}.
\end{proof}

Corollary \ref{cor-BM-no-critical-pts} implies the following rigidity theorem. 

\begin{thm} \label{thm-rigidity-no-critical-point}
Let $(M,g)$ be a complete, connected, asymptotically flat $3$-manifold without boundary, with finitely many ends.
If there exists  a  static potential  $f$ on $(M, g)$ which has no critical points,
then $(M, g)$ is  isometric  to either $(\R^3, g_0)$ or  a  spatial Schwarzschild manifold
$ ( \R^3 \setminus \{ 0  \},  ( 1 + \frac{  m}{ 2 | x | }  )^4 g_0 ) $ with $m>0$.
\end{thm}
\begin{proof}
If $f^{-1}(0)$ has no compact component, then $(M,g)$ is isometric to $(\R^3,g_0)$ by Proposition \ref{prop-integration} (ii)
(cf. the proof of  Theorem \ref{thm-simple-top}).
Next, suppose $f^{-1}(0)$ has a  compact component $\Sigma$.
Cutting $M$ along $\Sigma$, and let $ (\tilde{M}, \tilde{g} ) $ be the metric completion  of   $( M\setminus \Sigma, g)$.
Then either $\tilde{M}$  has two components  whose boundary is isometric to $\Sigma$,
or $\tilde{M}$ is  connected  with two boundary components that are isometric to $\Sigma$.
Applying   Corollary \ref{cor-BM-no-critical-pts} to each component of $(\tilde{M}, \tilde{g})$ shows
that $(\tilde{M}, \tilde{g})$ can not be connected, and hence has two components each of which is  isometric to a spatial Schwarzschild manifold
with positive mass outside its horizon. Since their boundaries are isometric, we conclude that
$(M, g)$ itself is isometric to a complete  spatial Schwarzschild manifold
with positive mass.
\end{proof}


\end{document}